\documentclass[10pt]{article}
\usepackage{amsmath,amssymb}
\usepackage{enumerate}
\usepackage[dvipsnames]{xcolor}

\allowdisplaybreaks[1]

\newtheorem{theorem}{Theorem}[section]

\newtheorem{lemma}[theorem]{Lemma}

\newtheorem{corollary}[theorem]{Corollary}
\newtheorem{definition}[theorem]{Definition}
\newtheorem{remark}[theorem]{Remark}
\newtheorem{example}[theorem]{Example}

\newcommand{\Z}{\mathbb{Z}}

\renewcommand{\ker}{\operatorname{Ker}}
\newcommand{\id}{\operatorname{id}}

\newcommand{\Sym}{\operatorname{Sym}}
\newcommand{\aut}{\operatorname{Aut}}

\newcommand{\ret}{\operatorname{Ret}}
\newcommand{\soc}{\operatorname{Soc}}

\newcommand{\Aut}{\operatorname{Aut}}

\newcommand{\Ret}{\operatorname{Ret}}
\newcommand{\gr}{\operatorname{gr}}

\newenvironment{proof}{\par\noindent{ Proof.}}{$\qed$\par\bigskip}
\newcommand{\qed}{\enspace\vrule  height6pt  width4pt  depth2pt}
\usepackage{color}

\begin{document}
\title{Indecomposable solutions of the Yang--Baxter equation of square-free cardinality\thanks{The first author was partially
supported by the grant MINECO PID2020-113047GB-I00 (Spain). }}
\author{F. Ced\'o \thanks{corresponding author} \and J. Okni\'{n}ski
}
\date{}

\maketitle


\vspace{30pt} \noindent Keywords: Yang--Baxter equation,
set-theoretic solution,
 indecomposable solution,  multipermutation solution, brace\\


\begin{abstract}
Indecomposable involutive non-degenerate set-theoretic solutions
$(X,r)$ of the Yang--Baxter equation of cardinality $p_1\cdots
p_n$, for different prime numbers $p_1,\ldots, p_n$, are studied.
It is proved that they are multipermutation solutions of level
$\leq n$. In particular, there is no simple solution of a
non-prime square-free cardinality. This solves a problem stated in
\cite{CO21} and provides a far reaching extension of several
earlier results on indecomposability of solutions. The proofs are
based on a detailed study of the brace structure on the
permutation group $\mathcal G(X,r)$ associated to such a solution.
It is proved that $p_1,\ldots, p_n$ are the only primes dividing
the order of $\mathcal{G}(X,r)$. Moreover, the Sylow
$p_i$-subgroups of $\mathcal{G}(X,r)$ are elementary abelian
$p_i$-groups and if $P_i$ denotes the Sylow $p_i$-subgroup of the
additive group of the left brace $\mathcal{G}(X,r)$, then there
exists a permutation $\sigma\in\Sym_n$ such that $P_{\sigma(1)},
\, P_{\sigma(1)}P_{\sigma(2)}, \dots ,$
$P_{\sigma(1)}P_{\sigma(2)}\cdots P_{\sigma(n)}$ are ideals of the
left brace $\mathcal{G}(X,r)$ and $\mathcal{G}(X,r)=P_1P_2\cdots
P_n$.  In addition, indecomposable solutions of cardinality
$p_1\cdots p_n$ that are multipermutation of level $n$ are
constructed, for every nonnegative integer $n$.
\end{abstract}

\section{Introduction}

The quantum Yang--Baxter equation (called YBE, for short),
originating from \cite{Baxter} and \cite{Yang}, has been an object
of a very extensive study in the recent decade. On one hand,
beyond its significance in mathematical physics, this is due to
the key role it plays in the foundations of the theory of quantum
groups and Hopf algebras, \cite{K}. On the other hand, the theory
developed for the so called set-theoretic solutions of this
equation (introduced by Drinfeld in \cite{drinfeld}) has lead to
several intriguing algebraic structures, and to a development of
new methods, that turned out to be very fruitful and also quite
fascinating on their own. In consequence, several deep relations
to a variety of different mathematical contexts have been
discovered, \cite{RumpSurvey}. A lot of effort has been
concentrated on the approach originating from the seminal papers
\cite{ESS}, of Etingof, Schedler and Soloviev, and \cite{GIVdB},
of Gateva-Ivanova and Van den Bergh, in the context of the so
called involutive non-degenerate solutions. One of the key lines
of attack is based on the very natural notion of indecomposable
solutions. This is because of the hope that the developed methods
might lead to the construction and classification of all
solutions. However, in general, indecomposable set-theoretic
solutions $(X,r)$ of the YBE turned out to be very difficult to
construct and classify,
 \cite{CCP,CPR,Jedl-Pilit-Zam,Jedl-Pilit-Zam2021,Ramirez,rump2020,SmokSmok}.
Crucial methods developed in this area are based, in particular,
on the permutation group $\mathcal G (X,r)\subseteq \Sym_X$
naturally associated to a solution $(X,r)$. This group carries a
natural structure of a brace, a powerful tool discovered by
Rump, \cite{Rump1,R07,RumpSurvey}, see also \cite{CedoSurvey}.
Another key idea coming from \cite{ESS}, strongly related to the
group $\mathcal G (X,r)$, is based on the retract $\Ret (X,r)$ of
a solution $(X,r)$, and on the induced notion of a
multipermutation solution. Roughly speaking, multipermutation
solutions have a controllable level of complexity, and are
considered easier to handle. This approach has attracted a lot of
attention  and allowed answering several questions, see for
example \cite{BCV,CJOComm,GIC}.

An important result of Rump established decomposability of  an
interesting class of solutions, \cite{Rump1}. Very recent results
in this area are based on a variety of ideas. In particular, on
the application of one-generator braces \cite{SmokSmok,Rump2020},
on the diagonal permutation methods \cite{Mora-Sastriq,
ramirez-vendramin}, the so called cabling process
\cite{lebed_and_co}, on arithmetical restrictions on the order of
the permutation group $\mathcal G (X,r)$ associated to the
solution \cite{Smokt}, or the approach based on general
constructions of braces, studied in \cite{BCJ} and applied in
\cite{Jedl-Pilit-Zam2021,Ramirez}. Indecomposable solutions that
happen to be multipermutational of small level have also been
extensively studied,
\cite{Jedl-Pilit-Zam,Jedl-Pilit-Zam2021,rump2022}.

In this paper we study finite indecomposable, involutive,
non-degenerate, set-theoretic solutions $(X,r)$ of the YBE of
square-free cardinality. One of our main motivations comes from
problems concerning the so called simple solutions, introduced in
\cite{V}. Namely, by \cite[Proposition~4.1]{CO21}, if $(X,r)$ is a
simple solution of the YBE and $|X|>2$, then $(X,r)$ is
indecomposable. Moreover, by \cite[Proposition~4.2]{CO21}, if
$|X|$ is not a prime, then $(X,r)$ is irretractable. In
\cite[Theorem 4.12]{CO21} it is proved that for any integers
$m,n>1$, there exists a simple solution $(X,r)$
     of the YBE of cardinality $m^2n$. Since indecomposable solutions of
     the YBE of prime cardinality
     are simple by \cite[Theorem 2.13]{ESS}, the following question, stated in
     \cite[Question 7.5]{CO21}, is natural:  Does there exist a finite simple
     solution $(X,r)$ of the YBE such that $|X|=p_1p_2\cdots p_n$, for
     $n>1$ and distinct primes $p_1,p_2,\dots ,p_n$?
     In this paper we answer this question in negative.
     Actually, the main result of this paper is much stronger, and it
     seems somewhat surprising. It reads as follows.
\begin{theorem}  \label{main-intro}
Let $(X,r)$ be an indecomposable, involutive, non-degenerate,
set-theoretic solution of the Yang--Baxter equation on the set $X$
of square-free cardinality. Then $(X,r)$ is a multipermutation
solution.
\end{theorem}

In particular, this result extends the description of all
indecomposable solutions of prime cardinality $p$, obtained in
\cite{ESS}. Namely, those are multipermutation solutions of level
$1$ and the associated permutation group $\mathcal G (X,r)$ is
cyclic of order $p$. This theorem  provides also a far reaching
extension of partial results of several authors on
indecomposability of $(X,r)$ in the case where $|X|=pq$ for two
different primes $p,q$,  \cite{CPR,lebed_and_co,Ramirez}.

Our second main result, Theorem~\ref{multi}, provides a
detailed information on the brace structure on $\mathcal G (X,r)$
for indecomposable solutions $(X,r)$ of square-free cardinality.
It is an important complement to Theorem~\ref{main-intro} and it
may become an essential tool in a further study of such
solutions.

     The paper is organized as follows. In Section 2 we recall all
     basic notions and results used in our approach. The structure of a left brace on the
     associated permutation group $\mathcal G
(X,r)$ plays a crucial role. In Section 3
     a detailed information on the structure of solvable transitive  permutation
     groups acting on a set of square-free cardinality is obtained.
      Since for every finite indecomposable solution $(X,r)$,  $\mathcal{G}(X,r)$
     is solvable and acts transitively on $X$, we deal with finite solvable transitive
     permutation groups of square-free degree, giving a transparent description of a chain
     of normal subgroups in Theorem \ref{squarefreedegree}.
     Then, in  Section 4, we study indecomposable
     multipermutation solutions of the YBE of cardinality $p_1\cdots p_n$ for
     different prime numbers $p_1,\dots ,p_n$.
     Theorem~\ref{multi}
     provides a deep insight into the structure of the corresponding
     left braces $\mathcal G(X,r)$. In particular, we prove that $p_1,\ldots, p_n$
     are the only primes dividing
     the order of $\mathcal{G}(X,r)$. Moreover, the Sylow
     $p_i$-subgroups of $\mathcal{G}(X,r)$ are elementary abelian
     $p_i$-groups and if $P_i$ denotes the Sylow $p_i$-subgroup of the
     additive group of the left brace $\mathcal{G}(X,r)$, then there
     exists a permutation $\sigma\in\Sym_n$ such that
     $P_{\sigma(1)}P_{\sigma(2)}\cdots P_{\sigma(i)}$ are ideals of the
     left brace $\mathcal{G}(X,r)$, for all $i=1,\dots , n$, and
     $\mathcal{G}(X,r)=P_1P_2\cdots
     P_n$.  In particular, Theorem~\ref{multi} answers Question 3.16 from
     \cite{ramirez-vendramin}  in affirmative in the case of solutions of square-free
     cardinality.
      Next, combining both results allows us to apply an inductive approach in order
      to prove our main result, stated in a more detailed way in Theorem~\ref{main}.
     Finally, a construction of a family of
     indecomposable solutions of cardinality $p_1\cdots p_n$ and multipermutation level $n$,
     for every non-negative integer $n$, is presented in Section 5.

\section{Preliminaries} \label{prelim}
Let $X$ be a non-empty set and  let  $r:X\times X \longrightarrow
X\times X$ be a map. For $x,y\in X$ we put $r(x,y) =(\sigma_x (y),
\gamma_y (x))$. Recall that $(X,r)$ is an involutive,
non-degenerate set-theoretic solution of the Yang--Baxter equation
if $r^2=\id$, all the maps $\sigma_x$ and $\gamma_y$ are bijective
maps from $X$ to itself and
  $$r_{12} r_{23} r_{12} =r_{23} r_{12} r_{23},$$
where $r_{12}=r\times \id_X$ and $r_{23}=\id_X\times \ r$ are maps
from $X^3$ to itself. Because $r^{2}=\id$, one easily verifies
that $\gamma_y(x)=\sigma^{-1}_{\sigma_x(y)}(x)$, for all $x,y\in
X$ (see for example \cite[Proposition~1.6]{ESS}).

\bigskip
\noindent {\bf Convention.} Throughout the paper a solution of the
YBE will mean an involutive, non-degenerate, set-theoretic
solution
of the Yang--Baxter equation. \\

It is well known, see for example \cite[Proposition
8,2]{CedoSurvey}, that a map $r(x,y) = (\sigma_x (y),
\sigma^{-1}_{\sigma_x(y)}(x))$, defined for a collection of
bijections $\sigma_x, x\in X$, of the set $X$, is a solution of
the YBE if and only if
$\sigma_{x}\sigma_{\sigma_{x}^{-1}(y)}=\sigma_{y}\sigma_{\sigma_{y}^{-1}(x)}$
for all $x,y\in X$.

A left brace is a set $B$ with two binary operations, $+$ and
$\circ$, such that $(B,+)$ is an abelian group (the additive group
of $B$), $(B,\circ)$ is a group (the multiplicative group of $B$),
and for every $a,b,c\in B$,
 \begin{eqnarray} \label{braceeq}
  a\circ (b+c)+a&=&a\circ b+a\circ c.
 \end{eqnarray}
In any left brace $B$  the neutral elements $0,1$ for the
operations $+$ and $\circ$ coincide. Moreover, there is an action
$\lambda\colon (B,\circ)\longrightarrow \aut(B,+)$, called the
lambda map of $B$, defined by $\lambda(a)=\lambda_a$ and
$\lambda_{a}(b)=-a+a\circ b$, for $a,b\in B$. We shall write
$a\circ b=ab$ and $a^{-1}$ will denote the inverse of $a$ for the
operation $\circ$, for all $a,b\in B$. A trivial brace is a left
brace $B$ such that $ab=a+b$, for all $a,b\in B$, i.e. all
$\lambda_a=\id$. The socle of a left brace $B$ is
$$\soc(B)=\{ a\in B\mid ab=a+b, \mbox{ for all
}b\in B \}.$$ Note that $\soc(B)=\ker(\lambda)$, and thus it is a
normal subgroup of the multiplicative group of $B$. The solution of
the YBE associated to a left brace $B$ is $(B,r_B)$, where
$r_B(a,b)=(\lambda_a(b),\lambda_{\lambda_a(b)}^{-1}(a))$, for all
$a,b\in B$ (see \cite[Lemma~2]{CJOComm}).

A left ideal of a left brace $B$ is a subgroup $L$ of the additive
group of $B$ such that $\lambda_a(b)\in L$, for all $b\in L$ and
all $a\in B$. An ideal of a left brace $B$ is a normal subgroup
$I$ of the multiplicative group of $B$ such that $\lambda_a(b)\in
I$, for all $b\in I$ and all $a\in B$. Note that
\begin{eqnarray}\label{addmult1}
ab^{-1}&=&a-\lambda_{ab^{-1}}(b)
\end{eqnarray}
 for all $a,b\in B$, and
    \begin{eqnarray} \label{addmult2}
     &&a-b=a+\lambda_{b}(b^{-1})= a\lambda_{a^{-1}}(\lambda_b(b^{-1}))= a\lambda_{a^{-1}b}(b^{-1}),
     \end{eqnarray}
for all $a,b\in B$. Hence, every left ideal $L$ of $B$ also is a
subgroup of the multiplicative group of $B$, and every  ideal $I$
of a left brace $B$ also is a subgroup of the additive group of
$B$. For example,  every Sylow subgroup of the additive group of a
finite left brace $B$ is a left ideal of $B$. Consequently, if a
Sylow subgroup of the multiplicative group $(B,\circ)$ is normal,
then it must be an ideal of the finite left brace $B$. It is also
known that $\soc(B)$ is an ideal of the left brace $B$ (see
\cite[Proposition~7]{R07}). Note that, for every ideal $I$ of $B$,
$B/I$ inherits a natural left brace structure.

Note that if $L_1$ and $L_2$ are left ideals of a left brace $B$,
then $L_1+L_2$ is a left ideal of $B$ and
    $$L_1L_2=L_1+L_2=L_2+L_1=L_2L_1.$$
     Moreover, if $L_1\cap L_2=\{ 1\}$, then every element $g\in L_1L_2$ is
     uniquely written as $g=g_1g_2$ with $g_1\in L_1$ and  $g_2\in L_2$.

A homomorphism of left braces is a map $f\colon B_1\longrightarrow
B_2$, where $B_1,B_2$ are left braces, such that $f(a b)=f(a)
f(b)$ and $f(a+b)=f(a)+f(b)$, for all $a,b\in B_1$. Note that the
kernel $\ker(f)$ of a homomorphism of left braces $f\colon
B_1\longrightarrow B_2$ is an ideal of $B_1$.

Let $B$ be a left brace. Suppose that $I$ is an ideal of $B$ and
$L$ is a left ideal of $B$ such that $I\cap L=\{0\}$ and $B=IL$.
Note that if $a\in I$ and $b\in L$, then, by \cite[Lemma
2]{CJOComm},
$bab^{-1}b=ba=\lambda_b(a)\lambda^{-1}_{\lambda_b(a)}(b)$ and $bb^{-1}ab=ab=\lambda_a(b)\lambda^{-1}_{\lambda_a(b)}(a)$. Since
$I\cap L=\{0\}$, we have that
\begin{equation}\label{semidirect1}
\lambda_a(b)=b\quad\mbox{and}\quad \lambda_b(a)=bab^{-1}
\end{equation}
for all $a\in I$ and $b\in L$. By the first equality of
(\ref{semidirect1}), we have that $ab=a+b$ for all $a\in I$ and
$b\in L$. Note that the multiplicative group of $B$ is the (inner)
semidirect product of the multiplicative groups of $I$ and $L$.
Furthermore,
\begin{equation}\label{semidirect2}
    a_1b_1+a_2b_2=a_1+b_1+a_2+b_2=(a_1+a_2)(b_1+b_2)
\end{equation}
for all $a_1,a_2\in I$ and $b_1,b_2\in L$. Let $\alpha\colon
(L,\circ)\longrightarrow \Aut(I,+,\circ)$ be the map defined by
$\alpha(b)(a)=\lambda_b(a)$ for all $a\in I$ and $b\in L$. By
(\ref{semidirect1}), $\alpha$ is well defined. Hence $\alpha$ is a
homomorphism of groups. By \cite[Section 3]{CedoSurvey} the
semidirect product of the left braces $I$ and $L$ via $\alpha$ is
the left brace $I\rtimes_{\alpha}L$, where the multiplicative
group is the semidirect product of the multiplicative groups of
$I$ and $L$ via $\alpha$ and the addition is defined by
$$(a_1,b_1)+(a_2,b_2)=(a_1+a_2,b_1+b_2)$$
for all $a_1,a_2\in I$ and $b_1,b_2\in L$. Now it is easy to see
that the map $f\colon I\rtimes_{\alpha}L\longrightarrow B$ defined
by $f(a,b)=ab$, for all $a\in I$ and $b\in L$, is an isomorphism
of left braces. For this reason it is said that the left brace $B$
is the inner semidirect product of the ideal $I$ and the left
ideal $L$.

Recall that if $(X,r)$ is a solution of the YBE, with
$r(x,y)=(\sigma_x(y),\gamma_y(x))$, then its structure group
$G(X,r)=\gr(x\in X\mid xy=\sigma_x(y)\gamma_y(x)$ for all $x,y\in
X)$ has a natural structure of a left brace such that
$\lambda_x(y)=\sigma_x(y)$, for all $x,y\in X$. The additive group
of $G(X,r)$ is the free abelian group with basis $X$. The
permutation group $\mathcal{G}(X,r)=\gr(\sigma_x\mid x\in X)$ of
$(X,r)$ is a subgroup of the symmetric group $\Sym_X$ on $X$.  The
map $x\mapsto \sigma_x$, from $X$ to $\mathcal{G}(X,r)$ extends to
a group homomorphism $\phi: G(X,r)\longrightarrow
\mathcal{G}(X,r)$ and $\ker(\phi)=\soc(G(X,r))$. Hence there is a
unique structure of a left brace on $\mathcal{G}(X,r)$ such that
$\phi$ is a homomorphism of left braces; this is the natural
structure of a left brace on $\mathcal{G}(X,r)$. In particular,
\begin{equation*}
    \sigma_{\sigma_x(y)}=\phi(\sigma_x(y))=\phi(\lambda_x(y))=\lambda_{\phi(x)}(\phi(y))=\lambda_{\sigma_x}(\sigma_y)
\end{equation*}
for all $x,y\in X$. An easy consequence is \cite[Lemma
2.1]{CJOprimit}, which says that
\begin{equation}\label{Lemma 2.1}
    \lambda_{g}(\sigma_y)=\sigma_{g(y)}
\end{equation}
for all $g\in\mathcal{G}(X,r)$ and all $y\in X$.

Let $(X,r)$ and $(Y,s)$ be solutions of the YBE. We write
$r(x,y)=(\sigma_x(y),\gamma_y(x))$ and
$s(t,z)=(\sigma'_t(z),\gamma'_z(t))$, for all $x,y\in X$ and $t,z\in
Y$. A homomorphism of solutions $f\colon (X,r)\longrightarrow (Y,s)$
is a map $f\colon X\longrightarrow Y$ such that
$f(\sigma_x(y))=\sigma'_{f(x)}(f(y))$ and
$f(\gamma_y(x))=\gamma'_{f(y)}(f(x))$, for all $x,y\in X$. Since
$\gamma_y(x)=\sigma^{-1}_{\sigma_x(y)}(x)$ and
$\gamma'_z(t)=(\sigma')^{-1}_{\sigma'_t(z)}(t)$, it is clear that
$f$ is a homomorphism of solutions if and only if
$f(\sigma_x(y))=\sigma'_{f(x)}(f(y))$, for all $x,y\in X$.

Note  that, by de defining relations of the structure group, every
homomorphism of solutions $f\colon (X,r)\longrightarrow (Y,s)$
extends to a unique homomorphism of groups $f\colon
G(X,r)\longrightarrow G(Y,s)$ that we also denote by $f$, and  it
is easily checked that it also is a homomorphism of left braces
and induces a homomorphism of left braces $\bar f\colon
\mathcal{G}(X,r)\longrightarrow\mathcal{G}(Y,s)$.

In \cite{ESS}, Etingof, Schedler and Soloviev introduced the
retract relation on solutions  $(X,r)$ of the YBE. This is the
binary relation $\sim$ on $X$ defined by $x\sim y$ if and only if
$\sigma_x=\sigma_y$. Then, $\sim$ is an equivalence relation and
$r$ induces a solution $\overline{r}$ on the set
$\overline{X}=X/{\sim}$. The retract of the solution $(X,r)$ is
$\Ret(X,r)=(\overline{X},\overline{r})$. Note that the natural map
$f\colon X\longrightarrow \overline{X}:x\mapsto \bar x$ is an
epimorphism of solutions from $(X,r)$ onto $\Ret(X,r)$.

Recall that a solution $(X,r)$ is  said to be irretractable if
$\sigma_x\neq \sigma_y$ for all distinct elements $x,y\in X$, that
is $(X,r)=\Ret(X,r)$; otherwise the solution $(X,r)$ is
retractable. Define $\ret^1(X,r)=\Ret(X,r)$ and, for every integer
$n>1$, $\Ret^{n}(X,r)=\Ret(\Ret^{n-1}(X,r))$. A solution $(X,r)$
is said to be a multipermutation solution if there exists a
positive integer $n$ such that $\Ret^n(X,r)$ has cardinality $1$.
In this case, the least positive integer $n$ such that
$\Ret^n(X,r)$ has cardinality $1$ is called the multipermutation
level of $(X,r)$.

\section{Permutation groups of indecomposable solutions}\label{sec4}
Let $(X,r)$ be a solution of the YBE. We say that $(X,r)$ is
indecomposable if $\mathcal{G}(X,r)$ acts transitively on $X$.
We write $r(x,y)=(\sigma_x(y),\gamma_y(x))$.

\begin{definition}
    A solution $(X,r)$ of the YBE is simple if $|X|>1$ and for every
    epimorphism $f:(X,r) \longrightarrow (Y,s)$ of solutions either $f$
    is an isomorphism or $|Y|=1$.
\end{definition}

In this context, the following result (Proposition~4.1 in
\cite{CO21} and Lemma~3.4 in \cite{CO22}) is crucial.

\begin{lemma}  \label{simple indec}
Assume that $(X,r)$ is a simple solution of the YBE. Then it is
indecomposable if $|X|>2$ and it is irretractable if $|X|$ is not
a prime number.
\end{lemma}

The following observation is a generalization of a result about
indecomposable cycle sets obtained in \cite{CCP}.

\begin{lemma}\label{CCP}
Assume that $f:(X,r) \longrightarrow (Y,s)$ is an epimorphism of
solutions. Let $Y=\{ y_1,\ldots ,y_n\}$. Then $X_{i}=
f^{-1}(y_{i}), i=1,\ldots, n$, is a system of imprimitivity for
the action of the permutation group ${\mathcal G}(X,r)$ on $X$.
Moreover, if $x,x'\in X_{i}$ for some $i$, then $\sigma_x$ and
$\sigma_{x'}$ determine the same permutation of blocks
$X_1,\ldots, X_n$.  In particular, if $(X,r)$ is indecomposable,
then $(Y,s)$ is indecomposable and $|X_i|=|X_j|$ for all $i,j$.
\end{lemma}
\begin{proof}
If $x,x'\in X_{i}$ then we know that $f(x)=f(x')$. Hence, if $z\in
X$ then we get
$$f(\sigma_z(x))=\sigma'_{f(z)}(f(x))=\sigma'_{f(z)}(f(x'))=f(\sigma_z(x')).$$
Therefore $\sigma_z(x),\sigma_z(x')\in X_{j}$ for some $j$. So
$\sigma_z$ determines a permutation $\sigma'_{z}$ of the subsets
$X_1, \ldots X_n$ and the first assertion follows. Now, let $y\in
X$. Then
$$f(\sigma_x(y))=\sigma'_{f(x)}(f(y))=\sigma'_{f(x')}(f(y))=f(\sigma_{x'}(y)).$$
Hence $\sigma_x(y),\sigma_{x'}(y)\in X_{k}$ for some $k$, which
proves the second assertion.

Assume that $(X,r)$ is indecomposable. This means that
$\mathcal{G}(X,r)$ is a transitive subgroup of $\Sym_X$. Hence,
for every $x\in X_i$ and $z\in X_j$, there exists $g\in
\mathcal{G}(X,r)$ such that $g(x)=y$. By the first assertion,
$g(X_i)=X_j$. Therefore $|X_i|=|X_j|$. Furthermore, if
$\bar{f}\colon \mathcal{G}(X,r)\longrightarrow \mathcal{G}(Y,r)$
is the homomorphism of groups induced by $f$, we have that
$\bar{f}(g)(y_i)=y_j$. Hence $(Y,s)$ is indecomposable and the
result follows.
\end{proof}

In this section we study indecomposable solutions $(X,r)$ of the
YBE such that $|X|=p_1\cdots p_n$, where $n>1$ and  $p_1, \dots
,p_n$ are $n$ distinct prime numbers. By \cite[Theorem
3.1]{CJOprimit} $(X,r)$ is not primitive, that is
$\mathcal{G}(X,r)$ is not a primitive subgroup of $\Sym_X$. Thus
$\mathcal{G}(X,r)$ is a transitive imprimitive group of degree
$p_1\cdots p_n$. By \cite[Theorem 2.15]{ESS} $G(X,r)$ is solvable
and hence $\mathcal{G}(X,r)$ is also solvable.

The following remark is part of a result that Huppert attributes
to Galois \cite[II, 3.2 Satz]{H}.

\begin{remark}\label{Galois}
    {\rm Let $G$ be a solvable primitive group of prime degree $p$.
        Then the stabilizer in $G$ of one point is a maximal subgroup $M$ of index $p$
        in $G$. Furthermore, $\bigcap_{g\in G}gMg^{-1}=\{ 1\}$. Let $P$ be a minimal normal
        subgroup of $G$. Then $PM=G$. Since $P$ is normal, $C_G(P)$ also is normal in $G$.
        Since $G$ is solvable, $P$ is abelian and $P\subseteq C_G(P)$. Note that
        $M\cap C_G(P)$ is normal in $G=PM$. But $M$ contains no non-trivial normal subgroup.
        Hence $M\cap C_G(P)=\{ 1\}$. Therefore $P=C_G(P)$ is a cyclic group of order $p$
        and $M$ acts faithfully by conjugation on $P$. Thus $G=PM$ is isomorphic to a
        subgroup of $\Z/(p)\rtimes \Aut(\Z/(p))$.}
\end{remark}

Our next result provides a convenient sufficient condition for
retractability of a solution $(X,r)$ in terms of a purely group
theoretic property of the multiplicative structure of the left
brace $\mathcal{G}(X,r)$.

\begin{lemma}\label{key}
    Let $(X,r)$ be a finite solution of the YBE. Suppose that $\mathcal{G}(X,r)$ has an
    abelian normal Sylow $p$-subgroup for some prime divisor $p$ of $|\mathcal{G}(X,r)|$.
    Then $(X,r)$ is retractable.
\end{lemma}

\begin{proof}
    Let $T$ be an abelian normal Sylow $p$-subgroup of $\mathcal{G}(X,r)$ for some prime divisor $p$ of $|\mathcal{G}(X,r)|$. Since the
    Sylow subgroups of the additive
    group of the left brace $\mathcal{G}(X,r)$ are left ideals, it follows that
     $T$ is an ideal of the left brace $\mathcal{G}(X,r)$.  Let $C$ be the Hall
    $p'$-subgroup of the additive group of the left brace $\mathcal{G}(X,r)$. We have
    that $\mathcal{G}(X,r)=TC$ is
    the inner semidirect product (as left braces) of
    the ideal $T$ and the left ideal $C$. By using the structure of the semidirect
    product, we know that $t+c= tc$ for $t\in T, c\in C$, and consequently
    \begin{equation}
        \label{lambdaT} \lambda_t(t_1c_1)=-t+tt_1c_1=-t+tt_1+c_1=(-t+tt_1)c_1=\lambda_t(t_1)c_1
    \end{equation}
    for all $t,t_1\in T$ and $c_1\in C$. Since $T$ is a finite
    non-zero left brace with abelian multiplicative group, by
    \cite[Proposition 3]{CJOComm}, $\soc(T)\neq \{\id\}$. Let $t\in
    \soc(T)\setminus\{\id\}$. There exists $x\in X$ such that
    $t(x)\neq x$. Let $t_x\in T$ and $c_x\in C$ be such that
    $\sigma_x=t_xc_x$. By (\ref{Lemma 2.1}) and
    (\ref{lambdaT}), we have that
    $$\sigma_{t(x)}=\lambda_t(\sigma_x)=\lambda_t(t_xc_x)=\lambda_t(t_x)c_x=t_xc_x=\sigma_x.$$
    Therefore $(X,r)$ is retractable and the result follows.
\end{proof}

In order to deal with the general case of solutions $(X,r)$ of
square-free cardinality, it is convenient to start with a
description of the associated permutation groups.

\begin{theorem}\label{squarefreedegree}
    Let $p_1,\dots ,p_n$ be $n$ distinct prime numbers. Let $G$ be a solvable
    transitive subgroup of $\Sym_X$, where $X$ is a set and $|X|=p_1\cdots p_n$.
    Then there exist $\sigma\in\Sym_n$ and normal subgroups
    $$\{ 1\}=K_0\subseteq T_1\subseteq K_1\subseteq T_2\subseteq  K_2\subseteq \dots\subseteq T_n\subseteq  K_n=G$$
    of $G$ such that for every $i=1,\dots ,n$,
    \begin{itemize}
        \item[(i)]   $K_i/K_{i-1}$ is isomorphic to a
    subgroup of
    $$\left(\Z/(p_{\sigma(i)})\rtimes\Aut(\Z/(p_{\sigma(i)}))\right)^{\frac{p_1\cdots p_n}{p_{\sigma(1)}\cdots p_{\sigma(i)} }},$$
    $p_{\sigma(i)}$ is a divisor of $|K_i/K_{i-1}|$ and
    $T_i/K_{i-1}$ is the Sylow $p_{\sigma(i)}$-subgroup of $K_i/K_{i-1}$,
    \item[(ii)] $K_i=\{ g\in G\mid g(T_i(x))=T_i(x)$ for all $x\in X\}$, where $T_i(x)=\{ t(x)\mid t\in T_i\}$,
    \item[(iii)] $\mathcal{S}_i=\{ K_i(x)\mid x\in X\}$ is a system of imprimitivity of $G$  and $|\mathcal{S}_i|=\frac{p_1\cdots p_n}{p_{\sigma(1)}\cdots p_{\sigma(i)}}$.
    \end{itemize}
\end{theorem}

\begin{proof}
      We shall prove the result by induction on $n$. For $n=1$, the result follows
    by Remark \ref{Galois}, taking as $T_1$ the Sylow $p_1$-subgroup of $G$, because every transitive group of prime degree is primitive.

    Suppose that $n>1$ and the result holds for $n-1$. Since $G$ is solvable, if $T$ is
    a minimal nontrivial normal subgroup of $G$, then $T$ is an elementary abelian
    $t$-subgroup for some prime divisor $t$ of $|G|$, and the orbits of $T$ in $X$
    form a system of imprimitivity of $G$. Hence $t\in \{ p_1,\dots ,p_n\}$.
    We may assume that $t=p_1$. Let $\mathcal{S}=\{ X_i\mid i=1,\dots ,p_2\cdots p_{n}\}$
    be the system of imprimitivity of $G$, where the $X_i=T(x_i)=\{ t(x_i)\mid t\in T\}$,
    for some $x_i\in X$, are the different orbits of $T$ in $X$.
    Let $G_{X_i}=\{ g\in G\mid g(X_i)=X_i\}$ for all $i=1,\dots ,p_2\cdots p_{n}$. Let
    $$K_1=\bigcap_{i=1}^{p_2\cdots p_{n}}G_{X_i}.$$
    Let $\psi\colon G\longrightarrow \Sym_{\mathcal{S}}$ be the map defined by
    $\psi(g)(T(x))=T(g(x))$ for all $g\in G$ and $x\in X$. It is easy to see that $\psi$
    is a homomorphism of groups and $\psi(G)$ is a transitive subgroup of
    $\Sym_{\mathcal{S}}$. Note that $\ker(\psi)=K_1$. By the inductive hypothesis,
    there exist $\sigma\in \Sym_{\{ 2,\dots ,n\}}$ and normal subgroups
    $$\ker(\psi)=K_1\subseteq T_2\subseteq K_2\subseteq \cdots \subseteq T_n\subseteq K_n=G$$
    of $G$ such that for every $i=2,\dots ,n$,
    \begin{itemize}
        \item[(i)]  $K_i/K_{i-1}$ is isomorphic to a
        subgroup of
        $$\left(\Z/(p_{\sigma(i)})\rtimes\Aut(\Z/(p_{\sigma(i)}))\right)^{\frac{p_2\cdots p_n}{p_{\sigma(2)}\cdots p_{\sigma(i)} }},$$
        $p_{\sigma(i)}$ is a divisor of $|K_i/K_{i-1}|$ and $T_i/K_{i-1}$ is the Sylow $p_{\sigma(i)}$-subgroup of $K_i/K_{i-1}$,
        \item[(ii)] $K_i=\{ g\in G\mid g(T_i(x))=T_i(x)$ for all $x\in X\}$, where $T_i(x)=\{ t(x)\mid t\in T_i\}=\{ g(x)\mid g\in K_i\}=K_i(x)$,
        \item[(iii)] $\mathcal{S}_i=\{ K_i(x)\mid x\in X\}$ is a system of imprimitivity of $G$  and $|\mathcal{S}_i|=\frac{p_2\cdots p_n}{p_{\sigma(1)}\cdots p_{\sigma(i)}}$.
    \end{itemize}
    Note that $K_1$ is a normal subgroup
    of $G$ such that $T\subseteq K_1$. Hence $p_1$ is a divisor of $|K_1|$.
     Since $T\subseteq K_1$, $K_1$ acts transitively on each $X_i$.
    Hence the image of the  homomorphism $\psi_i\colon K_1\longrightarrow \Sym_{X_i}$,
    defined by $\psi_i(g)(y)=g(y)$ for all $g\in K_1$ and $y\in X_i$, is a transitive
    subgroup of $\Sym_{X_i}$, and it is primitive because $|X_i|=p_1$ is prime.
    Hence $K_1/\ker(\psi_i)$ is a primitive group of degree $p_1$. By Remark \ref{Galois},
    $$K_1/\ker(\psi_i)\cong \Z/(p_1)\rtimes \langle \eta_i\rangle,$$
    where $\eta_i\in\Aut(\Z/(p_1))$. Note that $\bigcap_{i=1}^{p_2\cdots p_{n}}\ker(\psi_i)=\{ 1\}$.
    Hence $K_1$ is isomorphic to a subgroup of the direct product
    $$\prod_{i=1}^{p_2\cdots p_{n}}\left(\Z/(p_1)\rtimes \langle\eta_i\rangle\right).$$
     Let $T_1$ be the Sylow $p_1$-subgroup of $K_1$. Note that $T_1$ is characteristic in
     $K_1$, and thus it is normal in $G$. Furthermore,
     $T\subseteq T_1$ and $T(x)=T_1(x)=K_1(x)$ for all $x\in X$.
     Therefore the result follows by induction.
\end{proof}

\begin{remark}\label{KTP}
     Let $G$ be a solvable transitive subgroup of $\Sym_X$,
     where $X$ is a set of cardinality $p_1\cdots p_n$ and $p_1,\dots ,p_n$ are $n$
     distinct prime numbers. Let $T_i,K_i$ be the normal subgroups of $G$
     appearing in the statement of Theorem \ref{squarefreedegree}. Let $P_i$ be a Sylow
     $p_i$-subgroup of $G$. Then
        $$K_i(x)=K_i(T_i(x))=T_i(x) \quad\mbox{and}\quad T_i=(T_i\cap P_i)K_{i-1}$$
        for all $x\in X$ and all $i\geq 1$.
\end{remark}

\section{Indecomposable solutions of square-free cardinality.}

In this section we study the structure of the left brace
$\mathcal{G}(X,r)$ for indecomposable multipermutation solutions
$(X,r)$ of square-free cardinality. This will allow us to prove
that every indecomposable solution of the YBE of square-free
cardinality is a multipermutation solution. In particular, this
answers Question 7.5 from \cite{CO21} in the negative.

\begin{theorem}\label{multi}
Let $n$ be a positive integer. Let $p_1,\dots ,p_n$ be $n$
distinct prime numbers. Assume that $(X,r)$ is an indecomposable
multipermutation solution of the YBE of cardinality
 $p_1\cdots p_n$. Let $P_i$ be the Sylow $p_i$-subgroup of the additive group of
 the left brace $\mathcal{G}(X,r)$, for $i=1,\dots ,n$. Then the following conditions hold.
\begin{itemize}
    \item[(i)] $\soc(\mathcal{G}(X,r))$ is a Hall $\pi$-subgroup of the additive group of the left brace $\mathcal{G}(X,r)$ for some non-empty subset $\pi$ of $\{ p_1,\dots ,p_n\}$.
    \item[(ii)]  $\mathcal{G}(X,r)=P_1\cdots P_n$.
    \item[(iii)] $P_i$ is a trivial brace over an elementary abelian $p_i$-group for every $i=1,\dots ,n$.
    \item[(iv)] There exists a permutation $\sigma\in\Sym_n$ such that
    $$P_{\sigma(1)}, \, P_{\sigma(1)}P_{\sigma(2)}, \dots , P_{\sigma(1)}P_{\sigma(2)}\cdots P_{\sigma(n)}$$
    are ideals of the left brace $\mathcal{G}(X,r)$, $P_{\sigma(1)}\subseteq\soc(\mathcal{G}(X,r))$ and $$(P_{\sigma(1)}\cdots P_{\sigma(i)})/(P_{\sigma(1)}\cdots P_{\sigma(i-1)})\subseteq \soc(\mathcal{G}(X,r)/(P_{\sigma(1)}\cdots P_{\sigma(i-1)}))$$
    for every $1<i\leq n$.
    \item[(v)] Let $K_i=\{ g\in \mathcal{G}(X,r)\mid g(P_{\sigma(1)}\cdots P_{\sigma(i)}(x))=P_{\sigma(1)}\cdots P_{\sigma(i)}(x) \mbox{ for all }x\in X \}$.
    Then $K_i=P_{\sigma(1)}\cdots P_{\sigma(i)}$ for every $1\leq i\leq n$.
\end{itemize}
\end{theorem}

\begin{proof}
    We shall prove the result by induction on $n$. For $n=1$, by \cite[Theorem 2.13]{ESS}, $\mathcal{G}(X,r)$ is the trivial brace of order $p_1$, and the result follows in this case.

    Suppose that $n>1$ and the result holds for all indecomposable multipermutation solutions of the YBE of cardinality $q_1\cdots q_m$, where $q_1,\dots ,q_m$ are $m$ distinct prime numbers and $m<n$ is a positive integer.
    Since $(X,r)$ is a multipermutation solution, by \cite[Theorem 4.13]{CJKVAV}
    the solutions $(G(X,r),r_G)$ and $(\mathcal{G}(X,r),r_{\mathcal{G}})$
    of the YBE associated to the left braces $G(X,r)$ and $\mathcal{G}(X,r)$, respectively,
    also are multipermutation solutions. Hence $\soc(\mathcal{G}(X,r))\neq\{\id\}$.
    Recall that the additive and multiplicative structures on the socle coincide.
    Let $p$ be a prime divisor of $|\soc(\mathcal{G}(X,r))|$. Let $P$ be the Sylow
    $p$-subgroup of $\soc(\mathcal{G}(X,r))$.
    Since $\soc(\mathcal{G}(X,r))$ is an abelian normal subgroup of the permutation
    group $\mathcal{G}(X,r)$, we have that $P$ is a normal subgroup of $\mathcal{G}(X,r)$.
    Since the orbits of $X$ by the action of $P$ form a system of imprimitivity of
    $\mathcal{G}(X,r)$, it follows that $p\in\{ p_1,\dots ,p_n\}$. We may assume that $p=p_1$. Let $P(x)=\{ g(x)\mid g\in P\}$ and let $\mathcal{S}=\{ P(x)\mid x\in X\}$. We have that $|P(x)|=p_1$ for all $x\in X$
 and $|\mathcal{S}|=p_2\cdots p_n$. Let $K=\{ g\in\mathcal{G}(X,r)\mid g(P(x))=P(x)$ for all $x\in X\}$.  Clearly $P\subseteq K$. Let
 $\varphi\colon \mathcal{G}(X,r)\longrightarrow \Sym_{\mathcal{S}}$ be the map defined by $\varphi(g)(P(x))=gP(x)=P(g(x))$ for all $x\in X$.
 Then $\ker(\varphi)=K$. Let $g\in K$ and $x\in X$. Since $g(P(x))=P(x)$, there exists $t\in P$ such that $g(x)=t(x)$. By (\ref{Lemma 2.1}),
     \begin{equation} \label{fixed}
     \lambda_g(\sigma_x)=\sigma_{g(x)}=\sigma_{t(x)}=\lambda_t(\sigma_x)=\sigma_x,
     \end{equation}
    because $t\in P\subseteq\soc(\mathcal{G}(X,r))$.
    Hence $K\subseteq\soc(\mathcal{G}(X,r))$.
    Let $\psi_x\colon K\longrightarrow \Sym_{P(x)}$ be the map defined
    by $\psi_x(g)(y)=g(y)$ for all $g\in K$ and $y\in P(x)$.
    Then $K/\ker(\psi_x)$ is an abelian transitive group of degree $p_1=|P(x)|$.
    Hence $K/\ker(\psi_x)\cong\Z/(p_1)$.
    Since $\bigcap_{P(x)\in\mathcal{S}}\ker(\psi_x)=\{\id\}$, we have
    that $K$ is isomorphic to a subgroup of $(\Z/(p_1))^{p_2\cdots p_n}$.
 As $K$ is a $p_1$-group, $K=P$ is
 an ideal of the left brace $\mathcal{G}(X,r)$.
 Let $s\colon \mathcal{S}\times\mathcal{S}\longrightarrow \mathcal{S}\times\mathcal{S}$
 be the map defined by $s(P(x),P(y))=(P(\sigma_x(y)),P(\sigma^{-1}_{\sigma_x(y)}(x)))$
 for all $ x,y\in X$. Let $g,h\in P$.
 In view of (\ref{fixed}), we have that
    $$\sigma_{g(x)}(h(y))=(\lambda_g(\sigma_x))(h(y))=\sigma_x(h(y))=\sigma_xh\sigma_x^{-1}(\sigma_x(y))\in P(\sigma_x(y)).$$
    Hence $s$ is well-defined. In order to see that $(\mathcal{S},s)$ is a solution of the YBE, it is enough to prove that

    \begin{equation}\label{P(x)}
        \sigma_{P(x)}\sigma_{\sigma_{P(x)}^{-1}(P(y))}=\sigma_{P(y)}\sigma_{\sigma_{P(y)}^{-1}(P(x))}\end{equation}
    for all $x,y\in X$, where $\sigma_{P(x)}(P(y))=P(\sigma_x(y))$. Let $x,y,z\in X$. We have that
    \begin{align*}
        \sigma_{P(x)}\sigma_{\sigma_{P(x)}^{-1}(P(y))}(P(z))&=\sigma_{P(x)}\sigma_{P(\sigma_x^{-1}(y))}(P(z))\\
        &=\sigma_{P(x)}(P(\sigma_{\sigma_x^{-1}(y)}(z)))=P(\sigma_x\sigma_{\sigma_x^{-1}(y)}(z)).
    \end{align*}
Since $\sigma_x\sigma_{\sigma_x^{-1}(y)}=\sigma_y\sigma_{\sigma_y^{-1}(x)}$, (\ref{P(x)}) follows. Hence $(\mathcal{S},s)$ is a solution of the YBE.
Let $\psi\colon (X,r)\longrightarrow (\mathcal{S},s)$ be the map defined by
$\psi(x)=P(x)$. Note that
$$\psi(\sigma_x(y))=P(\sigma_x(y))=\sigma_{P(x)}(P(y))=\sigma_{\psi(x)}(\psi(y))$$
for all $x,y\in X$. Hence $\psi$ is a surjective homomorphism of
solutions. Since $(X,r)$ is indecomposable, by Lemma \ref{CCP},
$(\mathcal{S},s)$ also is indecomposable. Furthermore $\psi$
induces a surjective homomorphism of left braces
$\overline{\psi}\colon\mathcal{G}(X,r)\longrightarrow
\mathcal{G}(\mathcal{S},s)$ such that
$\overline{\psi}(\sigma_x)=\sigma_{P(x)}$ for all $x\in X$. Since
$(X,r)$ is a multipermutation solution, by \cite[Lemma
4]{CJOComm}, $(\mathcal{S},s)$ is a multipermutation solution. Let
$g\in \mathcal{G}(X,r)$. There exist $x_1,\dots,x_k\in X$ such
that $g=\sigma_{x_1}\cdots\sigma_{x_k}$.  We have that
\begin{align*}gP(x)&=P(g(x))=P(\sigma_{x_1}\cdots\sigma_{x_k}(x))=\sigma_{P(x_1)}\cdots\sigma_{P(x_k)}(P(x))\\
    &=\overline{\psi}(\sigma_{x_1}\cdots\sigma_{x_k})(P(x))=\overline{\psi}(g)(P(x))\end{align*}
for all $x\in X$. Hence $\ker(\overline{\psi})=K=P$ is an ideal of
the left brace $\mathcal{G}(X,r)$. Let $\overline{P}_i$ be the
Sylow $p_i$-subgroup of the additive group of the left brace
$\mathcal{G}(\mathcal{S},s)$. Since $|\mathcal{S}|=p_2\cdots p_n$,
by the inductive hypothesis we have that
\begin{itemize}
    \item[(i)] $\soc(\mathcal{G}(\mathcal{S},s))$ is a Hall $\pi$-subgroup
    of the additive group of the left brace $\mathcal{G}(\mathcal{S},s)$ for some non-empty subset $\pi$ of $\{ p_2,\dots ,p_n\}$.
    \item[(ii)]  $\mathcal{G}(\mathcal{S},s)=\overline{P}_2\cdots \overline{P}_n$.
    \item[(iii)] $\overline{P}_i$ is a trivial brace over an elementary abelian $p_i$-group for all $i=2,\dots ,n$.
    \item[(iv)] There exists a permutation $\sigma\in\Sym_{\{ 2,\dots ,n\}}$ such that
    $$\overline{P}_{\sigma(2)}, \, \overline{P}_{\sigma(2)}\overline{P}_{\sigma(3)}, \dots , \overline{P}_{\sigma(2)}\overline{P}_{\sigma(3)}\cdots \overline{P}_{\sigma(n)}$$
    are ideals of the left brace $\mathcal{G}(\mathcal{S},s)$, $\overline{P}_{\sigma(2)}\subseteq\soc(\mathcal{G}(\mathcal{S},s))$ and $$(\overline{P}_{\sigma(2)}\cdots \overline{P}_{\sigma(i)})/(\overline{P}_{\sigma(2)}\cdots \overline{P}_{\sigma(i-1)})\subseteq \soc(\mathcal{G}(\mathcal{S},s)/(\overline{P}_{\sigma(2)}\cdots \overline{P}_{\sigma(i-1)}))$$
    for all $2<i\leq n$.
    \item[(v)] Let $\overline{K}_i=\{ g\in \mathcal{G}(\mathcal{S},s)\mid g(\overline{P}_{\sigma(2)}\cdots \overline{P}_{\sigma(i)}(A))=\overline{P}_{\sigma(2)}\cdots \overline{P}_{\sigma(i)}(A) \mbox{ for all }A\in \mathcal{S} \}$. Then $\overline{K}_i=\overline{P}_{\sigma(2)}\cdots \overline{P}_{\sigma(i)}$ for all $2\leq i\leq n$.
\end{itemize}
Hence $P=\ker(\overline{\psi})$ is a normal Sylow $p_1$-subgroup
of $\mathcal{G}(X,r)$, and thus $P=P_1$ is the Sylow
$p_1$-subgroup of the additive group of the left brace
$\mathcal{G}(X,r)$ and it is an ideal of this left brace.
Furthermore, $P_1\subseteq \soc(\mathcal{G}(X,r))$ and
$K_1=K=P=P_1$.

 Thus, we have proved that every Sylow $p$-subgroup of $\soc(\mathcal{G}(X,r))$
 also is a Sylow $p$-subgroup of $\mathcal{G}(X,r)$.
 Since the multiplicative group of $\soc(\mathcal{G}(X,r))$ is abelian, we have that
 $\soc(\mathcal{G}(X,r))$ is a Hall subgroup of $\mathcal{G}(X,r)$.
  Note that $P=P_1$ is a trivial brace over an elementary abelian $p_1$-group.
  Since $\mathcal{G}(\mathcal{S},s)\cong \mathcal{G}(X,r)/P_1$ as left braces,
  we have that $\overline{P}_i\cong P_i$ for all $i=2,\dots, n$ and
 $$P_1, \, P_1P_{\sigma(2)}, \dots , P_1P_{\sigma(2)}\cdots P_{\sigma(n)}$$
 are ideals of the left brace $\mathcal{G}(X,r)=P_1P_2\cdots P_n$. Furthermore
 $$(P_1P_{\sigma(2)}\cdots P_{\sigma(i)})/(P_1P_{\sigma(2)}\cdots P_{\sigma(i-1)})\subseteq \soc(\mathcal{G}(X,r)/(P_1P_{\sigma(2)}\cdots P_{\sigma(i-1)}))$$
 for all $1<i\leq n$. Let $i\geq 2$ and let $K_i=\{ g\in\mathcal{G}(X,r)\mid g(P_1P_{\sigma(2)}\cdots  P_{\sigma(i)}(x))=P_1P_{\sigma(2)}\cdots  P_{\sigma(i)}(x)\mbox{ for all }x\in X\}$.
 Since $P_1P_{\sigma(2)}\cdots  P_{\sigma(i)}(x)=P_{\sigma(2)}\cdots  P_{\sigma(i)}(P_1(x))$
 and $\mathcal{S}=\{ P_1(x)\mid x\in X\}$, we have that
 $$\overline{\psi}(K_i)=\{ g\in\mathcal{G}(\mathcal{S},s)\mid g(\overline{P}_{\sigma(2)}\cdots  \overline{P}_{\sigma(i)}(A))=\overline{P}_{\sigma(2)}\cdots  \overline{P}_{\sigma(i)}(A)\mbox{ for all }A\in \mathcal{S}\}=\overline{K}_i.$$
 Hence
 $$K_i=\overline{\psi}^{-1}(\overline{K}_i)=K_1P_{\sigma(2)}\cdots P_{\sigma(i)}=P_1P_{\sigma(2)}\cdots P_{\sigma(i)}$$
 for all $2\leq i\leq n$.
 Therefore the  result follows by induction.
\end{proof}

 The following result is an easy consequence of Theorem \ref{multi}.

\begin{corollary}\label{cmulti}
    Let $n$ be a positive integer. Let $p_1,\dots ,p_n$ be $n$ distinct prime numbers.
    Let $(X,r)$ be an indecomposable multipermutation solution of the YBE of
    cardinality $p_1\cdots p_n$. Then $(X,r)$ and the  solution associated to
    the left brace $\mathcal{G}(X,r)$ have multipermutation level $\leq n$.
\end{corollary}

    In order to prove our main result we need some preparatory
lemmas.

\begin{lemma}\label{lambdasum}
    Let $(X,r)$ be a solution of the YBE. Then in $G(X,r)$ we have that
    $$\lambda_{x_1+\dots +x_n}(y)=(\sigma_{x_1}+\dots +\sigma_{x_n})(y)$$
    for all $x_1,\dots ,x_n,y\in X$.
\end{lemma}

\begin{proof}
    We shall prove the result by induction on $n$. For $n=1$, clearly we have that
    $\lambda_{x_1}(y)=\sigma_{x_1}(y)$ by the definition of the lambda map in $G(X,r)$.
    Suppose that $n>1$ and that the result holds for $n-1$.
    Using the inductive hypothesis, the relation $a+b = a \lambda^{-1}_{a}(b)$ for $a,b\in G(X,r)$ and (\ref{Lemma 2.1}), we have
    \begin{align*}
        \lambda_{x_1+\dots +x_n}(y)&=\lambda_{(x_1+\dots +x_{n-1})\lambda_{x_1+\dots +x_{n-1}}^{-1}(x_n)}(y)\\
         &=\lambda_{x_1+\dots +x_{n-1}}\lambda_{\lambda_{x_1+\dots +x_{n-1}}^{-1}(x_n)}(y)\\
         &=(\sigma_{x_1}+\dots +\sigma_{x_{n-1}})\sigma_{(\sigma_{x_1}+\dots +\sigma_{x_{n-1}})^{-1}(x_n)}(y)\\
        &=((\sigma_{x_1}+\dots +\sigma_{x_{n-1}})\lambda_{(\sigma_{x_1}+\dots +\sigma_{x_{n-1}})^{-1}}(\sigma_{x_n}))(y)\\
         &=(\sigma_{x_1}+\dots +\sigma_{x_{n}})(y)
\end{align*}
for all $x_1,\dots ,x_n,y\in X$. Therefore the result follows by induction.
\end{proof}

 An important and quite
nontrivial step in our approach is based on a comparison of the
chain of normal subgroups coming from
Theorem~\ref{squarefreedegree} and the chain of ideals resulting
from Theorem~\ref{multi}. This will be crucial in the inductive
procedure used in the proof of the main theorem,
Theorem~\ref{main}.

\begin{lemma}\label{multi2}
    Let $n$ be a positive integer. Let $p_1,\dots ,p_n$ be $n$ distinct prime numbers. Assume that $(X,r)$ is an indecomposable multipermutation solution
    of the YBE of cardinality $p_1\cdots p_n$. Let $P_i$ be the Sylow $p_i$-subgroup of the additive group of the left brace $\mathcal{G}(X,r)$, for $i=1,\dots ,n$. Let
    $$\{ 1\}=K_0\subseteq T_1\subseteq K_1\subseteq T_2\subseteq K_2\subseteq \dots\subseteq T_n\subseteq K_n=\mathcal{G}(X,r)$$
    be normal subgroups of $\mathcal{G}(X,r)$ such that
    \begin{itemize}
        \item[(i)] $p_i$ is a divisor of $|K_i/K_{i-1}|$ and $T_i/K_{i-1}$ is the Sylow $p_i$-subgroup of $K_i/K_{i-1}$.
        \item[(ii)] $K_i=\{ g\in\mathcal{G}(X,r)\mid g(T_i(x))=T_i(x)$ for all $x\in X\}$, where $T_i(x)=\{ t(x)\mid t\in T_i\}$.
    \end{itemize}
Then
$$K_i=P_1\cdots P_i$$
is an ideal of the left brace $\mathcal{G}(X,r)$ and
$$\lambda_g(h)=h$$
for all $g\in K_i$, $h\in P_i\cdots P_n$ and all $1\leq i\leq n$.
\end{lemma}

\begin{proof}
    We shall prove the result by induction on $n$. For $n=1$, $|X|=p_1$ and, by \cite[Theorem 2.13]{ESS}, $\mathcal{G}(X,r)=K_1=P_1$ is the trivial brace of order $p_1$.
    Hence the result holds for $n=1$. Suppose that $n>1$ and that the result holds for every indecomposable multipermutation solution of cardinality $q_1\cdots q_m$,
    where $m<n$ and $q_1,\dots ,q_m$ are $m$ distinct prime numbers.

    By Theorem \ref{multi}, $P_i$ is a trivial brace, for every $i=1,\dots ,n$, and $\mathcal{G}(X,r)=P_1\cdots P_n$. Let $t\in T_1$ and $g\in P_j$ for $j>1$. We have that
    $$\lambda_t(g)\lambda^{-1}_{\lambda_t(g)}(t)=tg=gg^{-1}tg.$$
    Since $\lambda_t(g),g\in P_j$, $\lambda^{-1}_{\lambda_t(g)}(t)\in P_1$ and
    $g^{-1}tg\in T_1\subseteq P_1$, it follows that
    $\lambda_t(g)=g$. Since $P_1$ is a trivial brace,
    $$\lambda_t(h)=h$$
    for all $h\in \mathcal{G}(X,r)=P_1+\dots +P_n$.
    Thus $T_1\subseteq \soc(\mathcal{G}(X,r))$.
    By Theorem \ref{multi}, $\soc(\mathcal{G}(X,r))$ is a hall $\pi$-subgroup of the
    additive group of the left brace $\mathcal{G}(X,r)$. In particular, $P_1\subseteq \soc(\mathcal{G}(X,r))$. Hence $P_1$ is an ideal of $\mathcal{G}(X,r)$.
    As in the proof of Theorem \ref{multi}, we see that $K_1=P_1$.

    Let $\mathcal{S}=\{ K_1(x)\mid x\in X\}$. Since $K_1$ is a normal subgroup of $\mathcal{G}(X,r)$, $\mathcal{S}$ is a system of imprimitivity of $\mathcal{G}(X,r)$.
    Clearly $|K_1(x)|=|P_1(x)|=p_1$ for all $x\in X$. Hence $|\mathcal{S}|=p_2\cdots p_n$. Let $\psi\colon \mathcal{G}(X,r)\longrightarrow\Sym_{\mathcal{S}}$
    be the map defined by $\psi(g)(K_1(x))=K_1(g(x))$ for all $g\in\mathcal{G}(X,r)$ and all $x\in X$. Note that $\psi$ is a homomorphism of groups and $\ker(\psi)=K_1=P_1$.
    Let $s\colon \mathcal{S}\times\mathcal{S}\longrightarrow \mathcal{S}\times \mathcal{S}$ be the map defined by
    $$s(K_1(x),K_1(y))=(\sigma_{K_1(x)}(K_1(y)),\sigma^{-1}_{\sigma_{K_1(x)}(K_1(y))}(K_1(x)))$$
    for all $x,y\in X$, where $\sigma_{K_1(x)}=\psi(\sigma_x)$ for all $x\in X$.  Let $g\in K_1$. By (\ref{Lemma 2.1}) and in
    view of the fact that $g\in \soc(\mathcal{G}(X,r))$, we have
    $$\psi(\sigma_{g(x)})=\psi(\lambda_g(\sigma_x))=\psi(\sigma_x).$$
    Hence $s$ is well defined. We shall prove that $(\mathcal{S},s)$ is an indecomposable
    multipermutation solution of the YBE.
    Let $x,y\in X$.  By the definition of $\psi$ and of the permutations $\sigma_{K_{1}(x)}$, we have that
    \begin{align*}
        \sigma_{K_1(x)}\sigma_{\sigma^{-1}_{K_1(x)}(K_1(y))}&=\psi(\sigma_x)\sigma_{\psi(\sigma^{-1}_x)(K_1(y))}\\
            &=\psi(\sigma_x)\sigma_{K_1(\sigma^{-1}_x(y))}\\
            &=\psi(\sigma_x)\psi(\sigma_{\sigma^{-1}_x(y)})\\
            &=\psi(\sigma_x\sigma_{\sigma^{-1}_x(y)})
    \end{align*}
Since $\sigma_x\sigma_{\sigma^{-1}_x(y)}=\sigma_y\sigma_{\sigma^{-1}_y(x)}$, we get that
$$\sigma_{K_1(x)}\sigma_{\sigma^{-1}_{K_1(x)}(K_1(y))}=\sigma_{K_1(y)}\sigma_{\sigma^{-1}_{K_1(y)}(K_1(x))}.$$
Hence $(\mathcal{S},s)$ is a solution of the YBE. Let $f\colon (X,r)\longrightarrow (\mathcal{S},s)$ be the map defined by $f(x)=K_1(x)$ for all $x\in X$. We have that
\begin{align*}
    f(\sigma_x(y))&=K_1(\sigma_x(y))=\psi(\sigma_x)(K_1(y))\\
    &=\sigma_{K_1(x)}(K_1(y))=\sigma_{f(x)}(f(y))
\end{align*}
for all $x,y\in X$. Hence $f$ is an epimorphism of solutions. By Lemma \ref{CCP}, $(\mathcal{S},s)$ is indecomposable.
By \cite[Lemma 4]{CJOComm}, $(\mathcal{S},s)$ is a multipermutation solution. We know that $f$ induces an epimorphism of left braces
$\overline{f}\colon \mathcal{G}(X,r)\longrightarrow\mathcal{G}(\mathcal{S},s)$ such that $\overline{f}(\sigma_x)=\sigma_{K_1(x)}=\psi(\sigma_x)$ for all $x\in X$.
Hence $$\mathcal{G}(\mathcal{S},s)=\psi(\mathcal{G}(X,r))\cong\mathcal{G}(X,r)/K_1\cong P_2\cdots P_n$$
as left braces. Note that
$$\{ 1\}=\psi(K_1)\subseteq \psi(T_2)\subseteq \psi(K_2)\subseteq \dots\subseteq \psi(T_n)\subseteq \psi(K_n)=\mathcal{G}(\mathcal{S},s)$$
are normal subgroups of $\mathcal{G}(\mathcal{S},s)$. Since $\psi(K_i)/\psi(K_{i-1})\cong (K_i/K_1)/(K_{i-1}/K_1)\cong K_i/K_{i-1}$ and $\psi(T_i)/\psi(K_{i-1})\cong (T_i/K_1)/(K_{i-1}/K_1)\cong T_i/K_{i-1}$, we have that
$p_i$ is a divisor of $|\psi(K_i)/\psi(K_{i-1})|$ and $\psi(T_i)/\psi(K_{i-1})$ is the Sylow $p_i$-subgroup of $\psi(K_i)/\psi(K_{i-1})$ for all $i>1$.
Since $K_i=\{ g\in\mathcal{G}(X,r)\mid g(T_i(x))=T_i(x)$ for all $x\in X\}$, we have that
$\psi(K_i)=\{g\in \mathcal{G}(\mathcal{S},s)\mid g(\psi(T_i)(K_1(x)))=\psi(T_i)(K_1(x))$ for all $x\in X \}$ for all $i>1$. Hence, by the inductive hypothesis,
$\psi(K_i)=\psi(P_2)\cdots \psi(P_i)$
is an ideal of $\mathcal{G}(\mathcal{S},s)$ and
$$\lambda_g(h)=h$$
holds in $\mathcal{G}(S,s)$, for all $g\in \psi(K_i)$, $h\in
\psi(P_i)\cdots \psi(P_n)$ and $i>1$. Therefore $K_i=P_1P_2\cdots
P_i$ is a normal subgroup, and hence an ideal of
$\mathcal{G}(X,r)$ and for all $g\in K_i$, $h\in P_i\cdots P_n$
and $i>1$ there exists $k\in K_1$ such that
$$\lambda_g(h)=hk$$
holds in $\mathcal{G}(X,r)$. Since $k\in K_1\cap(P_i\cdots
P_n)=P_1\cap(P_i\cdots P_n)=\{ 1\}$, we have that
$\lambda_g(h)=h$. Therefore the result follows by induction.
\end{proof}

We are now ready for the main result of this paper.

\begin{theorem}  \label{main}
    Let $n$ be a positive integer. Let $p_1,\dots ,p_n$ be $n$ distinct prime numbers.
    Let $(X,r)$ be an indecomposable solution of the YBE of cardinality $p_1\cdots p_n$.
    Then $(X,r)$ is a multipermutation solution.  In particular, $(X,r)$ is not a simple
    solution if $n>1$.
\end{theorem}

\begin{proof}
    Suppose that the result is not true. Let $n$ be the smallest positive integer such that
    there exists an indecomposable irretractable solution $(X,r)$ of the YBE of cardinality
    $p_1\cdots p_n$ for some distinct prime numbers $p_1,\dots ,p_n$.
    By \cite[Theorem 2.13]{ESS}, $n>1$. For every $i=1,\dots ,n$, let $P_i$ be the Sylow
    $p_i$-subgroup of the additive group of the left brace $\mathcal{G}(X,r)$. Let $M$ be the
    Hall $\{ p_1,\dots,p_n\}'$-subgroup of the additive group of the left brace
    $\mathcal{G}(X,r)$. We have that $P_i$ and $M$ are left ideals of the left brace $\mathcal{G}(X,r)$,
    $$\mathcal{G}(X,r)=P_1+\dots +P_n+M=P_1\cdots P_nM,$$
    $P_iP_j=P_i+P_j=P_jP_i$ and $P_iM=P_i+M=MP_i$.
    For every $x\in X$ there exist unique $\eta_{i,x}\in P_i$ and $m_x\in M$ such that
    $$\sigma_x=\eta_{1,x}+\dots +\eta_{n,x}+m_x.$$
    By (\ref{Lemma 2.1}), we have that
    $$\sigma_{g(x)}=\lambda_g(\sigma_x)=\lambda_g(\eta_{1,x})+\dots \lambda_g(\eta_{n,x})+\lambda_g(m_x)$$
    for all $g\in \mathcal{G}(X,r)$ and $x\in X$. Hence
    \begin{equation} \label{lambda} \lambda_g(\eta_{i,x})=\eta_{i,g(x)} \quad\mbox{and}\quad \lambda_g(m_x)=m_{g(x)} \end{equation}
    for all $i=1,\dots ,n$ and $x\in X$.

    By Theorem \ref{squarefreedegree}, we may assume that there exist
    normal subgroups
    $$\{ 1\}=K_0\subseteq T_1\subseteq K_1\subseteq T_2\subseteq  K_2\subseteq \dots\subseteq T_n\subseteq  K_n=\mathcal{G}(X,r)$$
    of $\mathcal{G}(X,r)$ such that for every $i=1,\dots ,n$,
    \begin{itemize}
        \item[(i)]   $K_i/K_{i-1}$ is isomorphic to a
        subgroup of
        $$\left(\Z/(p_i)\rtimes\Aut(\Z/(p_i))\right)^{p_{i+1}\cdots p_n},$$
        $p_i$ is a divisor of $|K_i/K_{i-1}|$ and
        $T_i/K_{i-1}$ is the Sylow $p_{i}$-subgroup of $K_i/K_{i-1}$,
        \item[(ii)] $K_i=\{ g\in \mathcal{G}(X,r)\mid g(T_i(x))=T_i(x)$ for all $x\in X\}$, where $T_i(x)=\{ t(x)\mid t\in T_i\}$,
        \item[(iii)] $\mathcal{S}_i=\{ K_i(x)\mid x\in X\}$ is a system of imprimitivity
        of $\mathcal{G}(X,r)$ and $|\mathcal{S}_i|=p_{i+1}\cdots p_n$.
    \end{itemize}
In particular, $T_1$ is a nontrivial normal elementary abelian
$p_1$-subgroup and thus $T_1\subseteq P_1$.

 Let $t\in T_1$ and $g\in P_2+\dots +P_n+M$. We have that
            $$gg^{-1}tg=tg=\lambda_t(g)\lambda^{-1}_{\lambda_t(g)}(t).$$
            Since $g,\lambda_t(g)\in P_2+\dots + P_n+M$, $\lambda^{-1}_{\lambda_t(g)}(t)
            \in P_1$ and $g^{-1}tg\in T_1\subseteq P_1$, it follows that
            \begin{equation}\label{T1g}
           \lambda_t(g)=g
            \end{equation}
 for all $t\in T_1$ and $g\in P_2+\dots +P_n+M$.
        Let $s\colon \mathcal{S}_1\times \mathcal{S}_1\longrightarrow \mathcal{S}_1\times \mathcal{S}_1$ be the map defined by
        $$s(T_1(x),T_1(y))=(\sigma_{T_1(x)}(T_1(y)), \sigma^{-1}_{\sigma_{T_1(x)}(T_1(y))}(T_1(x))),$$
        where $\sigma_{T_1(x)}(T_1(y))=T_1((\sigma_x-\eta_{1,x})(y))$ for all $x,y\in X$.
Let $t_1,t_2\in T_1$. For every $x\in X$, let $\eta_x=\sigma_x-\eta_{1,x}$.
So
\begin{equation} \label{T1} T_1(\eta_x (y))= \sigma_{T_{1}(x)}(T_{1}(y)).
\end{equation}
Note that $\eta_x\in P_2+\dots +P_n+M$ and by
 (\ref{lambda}) we get
        \begin{align*}
            \eta_{t_1(x)}(t_2(y))
        &=\eta_{t_1(x)}t_2\eta_{t_1(x)}^{-1}\eta_{t_1(x)}(y)\\
        &=\eta_{t_1(x)}t_2\eta_{t_1(x)}^{-1}(\lambda_{t_1}(\eta_{x}))(y)\\
        &=\eta_{t_1(x)}t_2\eta_{t_1(x)}^{-1}\eta_{x}(y)\in T_1(\eta_{x}(y))
        \end{align*}
        for all $x,y\in X$. Hence $s$ is well-defined.

        We claim that $(\mathcal{S}_1,s)$ is an indecomposable solution of the YBE of cardinality $p_2\cdots p_n$.
        Let $k$ be a positive integer such that
        $$ k\equiv 0\mod |P_1|\quad\mbox{and}\quad  k\equiv 1 \mod |P_2+\dots +P_n+M|.$$
        Note that in the left brace $G(X,r)$,
        \begin{equation}\label{kk} \lambda_{kx}(ky)=k\lambda_{kx}(y)
        \end{equation}
(because $\lambda_{g}\in\Aut(G(X,r),+)$ for $g\in G(X,r)$)
        for all $x,y\in X$. Let $kX=\{ kx \mid x\in X\}$. Let $(G(X,r),r_G)$ be the solution of the YBE associated
        to the left brace $G(X,r)$. Let $r'\colon kX\times kX\longrightarrow kX\times kX$ be the restriction of
        $r_G$ to $kX\times kX$. Hence $(kX,r')$ is a solution of the YBE.
        Now we have  that
            \begin{align*}
                k\sigma_x&=k\eta_{1,x}+k\eta_{2,x}+\dots +k\eta_{n,x}+km_x\\
                &=\eta_{2,x}+\dots +\eta_{n,x}+m_x=\sigma_x-\eta_{1,x} =
                \eta_x,
                \end{align*}
        for all $x\in X$.
                Consequently,
        \begin{equation}\label{lambdasigma}
          T_1((k\sigma_x)(y))= \sigma_{T_{1}(x)}(T_1(y)).
        \end{equation}
         From Lemma \ref{lambdasum}, we know that
        \begin{equation}\label{ksigma2}
            \lambda_{kx}(y)=(k\sigma_x)(y)
        \end{equation}
        for all $x,y\in X$.
        To prove that $(\mathcal{S}_1,s)$ is a solution of the YBE it is enough to show
        that
        \begin{equation} \label{sigmaT2} \sigma_{T_1(x)}\sigma_{\sigma^{-1}_{T_1(x)}(T_1(y))}=\sigma_{T_1(y)}\sigma_{\sigma^{-1}_{T_1(y)}(T_1(x))}
            \end{equation}
    for all $x,y\in X$. Let $x,y,u\in X$.  Using  (\ref{lambdasigma}), (\ref{ksigma2}) and (\ref{kk})
      we have
        \begin{align*}
            \sigma_{T_1(x)}\sigma_{\sigma^{-1}_{T_1(x)}(T_1(y))}(T_1(u))&=\sigma_{T_1(x)}\sigma_{T_1((k\sigma_x)^{-1}(y))}(T_1(u))\\
            &=\sigma_{T_1(x)}(T_1(k\sigma_{(k\sigma_x)^{-1}(y)})(u))\\
            &=T_1((k\sigma_{x})(k\sigma_{(k\sigma_x)^{-1}(y)})(u))\\
            &=T_1(\lambda_{kx}\lambda_{k\lambda_{kx}^{-1}(y)}(u))\\
            &=T_1(\lambda_{kx}\lambda_{\lambda_{kx}^{-1}(ky)}(u))
        \end{align*}
        for all $x,y,u\in X$. Similarly we get
        $$\sigma_{T_1(y)}\sigma_{\sigma^{-1}_{T_1(y)}(T_1(x))}(T_1(u))=T_1(\lambda_{ky}\lambda_{\lambda_{ky}^{-1}(kx)}(u))$$
        for all $x,y,u\in X$. Since $(G(X,r),r_G)$ is a solution of the YBE, we have that
        $$\lambda_{kx}\lambda_{\lambda_{kx}^{-1}(ky)}=\lambda_{ky}\lambda_{\lambda_{ky}^{-1}(kx)}$$
        for all $x,y\in X$. Thus (\ref{sigmaT2}) follows.
        Hence $(\mathcal{S}_1,s)$ is a solution of the YBE.

        Note that the map $\psi\colon\mathcal{G}(X,r)\longrightarrow \Sym_{\mathcal{S}_1}$, defined by
        $\psi(g)(T_1(x))=T_1(g(x))$ for all $g\in \mathcal{G}(X,r)$ and $x\in X$, is a homomorphism of groups. Clearly $\ker(\psi)=K_1$.
        Let $x_0\in X$. By Remark \ref{KTP}, we have that
               \begin{align*}
                X&=K_n(x_0)=T_n(x_0)=((T_n\cap P_n)K_{n-1})(x_0)=(T_n\cap P_n)(T_{n-1}(x_0))\\
                &=\dots =((T_n\cap P_n)\cdots (T_2\cap P_2)K_1)(x_0)=((T_n\cap P_n)\cdots (T_2\cap P_2)T_1)(x_0).
                \end{align*}
        Then, for every $x\in X$, there exist $t_i\in T_i\cap P_i$, for $i=1,\dots ,n$,
        such that $x=t_n\cdots t_1(x_0)$.
         Hence $P_2+\dots +P_n=P_2 \cdots P_n$ acts transitively on $\mathcal{S}_1$, and thus $(\mathcal{S}_1,s)$ is indecomposable. Therefore the claim follows.

        By the minimality of $n$, $(\mathcal{S}_1,s)$ is a multipermutation solution.
        Note that
        $\psi(P_2+\dots +P_n+M)=\mathcal{G}(\mathcal{S}_1,s)$. Let $x,y,z\in X$. Then,
        using (\ref{T1}) and the definition of $\psi$,
        we get  $\psi(\eta_x)(T_1(z))=T_1(\eta_x(z))=\sigma_{T_1(x)}(T_1(z))$,
        and thus, from (\ref{lambda}) and (\ref{Lemma 2.1}), it follows that
         \begin{align*}\psi(\lambda_{\eta_x}(\eta_y))(T_1(z))&=\psi(\eta_{\eta_x(y)})(T_1(z))=T_1(\eta_{\eta_x(y)}(z))\\
            &=\sigma_{T_1(\eta_x(y))}(T_1(z))=\sigma_{\sigma_{T_1(x)}(T_1(y))}(T_1(z))\\
            &=\lambda_{\sigma_{T_1(x)}}(\sigma_{T_1(y)})(T_1(z))=(\lambda_{\psi(\eta_x)}(\psi(\eta_y)))(T_1(z)).
        \end{align*}
Hence the map $\psi'\colon P_2+\dots +P_n+M\longrightarrow \mathcal{G}(\mathcal{S}_1,s)$, defined by $\psi'(g)=\psi(g)$ for all $g\in P_2+\dots +P_n+M$, is an epimorphism of the associated solutions of these left braces. So $\psi'$ also is an epimorphism of left braces. Hence $\ker(\psi')=K_1\cap(P_2+\dots +P_n+M)$ is an ideal of $P_2+\dots +P_n+M$.

 Note that $\psi'(P_i)$ is the Sylow $p_i$-subgroup of the
additive group of the left
        brace $\mathcal{G}(\mathcal{S}_1,s)$. We write $\overline{P}_i=\psi'(P_i)$. By Theorem \ref{multi},
        $\mathcal{G}(\mathcal{S}_1,s)=\overline{P}_2\cdots \overline{P}_n$ and
        $\overline{P}_i$ is a trivial brace over an elementary abelian $p_i$-group,
        for all $i=2,\dots ,n$. In particular, $M\subseteq \ker(\psi')\subseteq K_1$.

   Moreover,
        \begin{align*}\{ 1\}&=\psi'(K_1\cap (P_2\cdots P_nM))\subseteq \psi'(T_2\cap (P_2\cdots P_nM))\subseteq  \psi'(K_2\cap (P_2\cdots P_nM))\\
            &\subseteq \dots\subseteq \psi'(T_n\cap (P_2\cdots P_nM))\subseteq  \psi'(K_n\cap (P_2\cdots P_nM))=\mathcal{G}(\mathcal{S}_1,s)
        \end{align*}
        are normal subgroups of $\mathcal{G}(\mathcal{S}_1,s)$ such that, for every $i>1$,
        \begin{itemize}
            \item[(i)] $p_i$ is a divisor of $|\psi'(K_i\cap (P_2\cdots P_nM))/\psi'(K_{i-1}\cap (P_2\cdots P_nM))|$ and $\psi'(T_i\cap (P_2\cdots P_nM))/\psi'(K_{i-1}\cap (P_2\cdots P_nM))$ is the Sylow $p_i$-subgroup of $\psi'(K_i\cap (P_2\cdots P_nM))/\psi'(K_{i-1}\cap (P_2\cdots
            P_nM))$,
            \item[(ii)] $\psi'(K_i\cap (P_2\cdots P_nM))=\{ g\in\mathcal{G}(\mathcal{S}_1,s)\mid g(\psi'(T_i\cap (P_2\cdots P_nM))(A))=\psi'(T_i\cap (P_2\cdots P_nM))(A)$ for all $A\in\mathcal{S}_1\}$.
        \end{itemize}
Therefore, by Lemma \ref{multi2},
$$\psi'(K_i\cap (P_2\cdots P_nM))=\overline{P}_2\cdots \overline{P}_i$$
is an ideal of the left brace $\mathcal{G}(\mathcal{S}_1,s)$ and
\begin{equation}\label{keylambda}
    \lambda_g(h)=h,
\end{equation}
for all $g\in \psi'(K_{i}\cap(P_2\cdots P_nM))$ and all $h\in
\overline{P}_i\cdots \overline{P}_n$.

Since $(X,r)$ is irretractable, by Lemma \ref{key}, $T_1\neq P_1$.
Hence there exists $i\geq 2$ such that $p_i>p_1$ and $P_1\cap
K_i\not\subseteq K_{i-1}$. Note that the structure of the group
$K_1$ implies that $|\overline{P}_i|=|P_i|$. Hence $P_i\cong
\overline{P}_i$ as left braces.
Let $h\in (P_1\cap K_i)\setminus K_{i-1}$. By Remark \ref{KTP},
$$K_i(x)=(T_i\cap P_i)\cdots (T_2\cap P_2)T_1(x)$$
for all $x\in X$. Hence for every $x\in X$ there exist $t_j\in
T_j\cap P_j$, for $j=1,\dots ,i$, such that $h(x)=t_i\cdots
t_1(x)$. By (\ref{lambda}), (\ref{T1g}) and (\ref{keylambda}) we
have that
        \begin{align*}\lambda_h(\eta_{i,x})&=\eta_{i,h(x)}=\eta_{i,t_i\cdots t_2t_1(x)}=\lambda_{t_i\cdots t_2t_1}(\eta_{i,x})\\
            &=\lambda_{t_i\cdots t_2}\lambda_{t_1}(\eta_{i,x})\\
            &=\lambda_{t_i\cdots t_2}(\eta_{i,x})=\eta_{i,x}k,
            \end{align*}
        for some $k\in K_1\cap P_i$. Since $p_i>p_1$, we have that $K_1\cap P_i=\{ 1\}$. Hence
        $$\lambda_h(\eta_{i,x})=\eta_{i,x} $$
        for all $x\in X$.
 Now it follows that $$h\eta_{i,x}=\lambda_h(\eta_{i,x})\lambda^{-1}_{\lambda_h(\eta_{i,x})}(h)=\eta_{i,x}\lambda^{-1}_{\lambda_h(\eta_{i,x})}(h).$$
Therefore
$\eta_{i,x}^{-1}h\eta_{i,x}=\lambda^{-1}_{\lambda_h(\eta_{i,x})}(h)\in
P_1\cap K_i$ (because $h\in P_1\cap K_i$) for all $x\in X$. Since
$P_i$ is a trivial brace, $P_i=\langle\eta_{i,x}\mid x\in
X\rangle_+=\langle\eta_{i,x}\mid x\in X\rangle$. As $h\in
(P_1\cap K_i)\setminus K_{i-1}$ was arbitrary, it follows that
$((P_1\cap K_i)K_{i-1})/K_{i-1}$ is a normal subgroup in
$(P_i(P_1\cap K_i)K_{i-1})/K_{i-1}$.
 Hence $((P_1\cap K_i)K_{i-1})/K_{i-1}$ is normal in $(T_i(P_1\cap
 K_i))/K_{i-1}$. Since $((P_1\cap K_i)K_{i-1})/K_{i-1}$ is a normal $p_1$-subgroup of
 $(T_i(P_1\cap
    K_i))/K_{i-1}$ and $T_i/K_{i-1}$ is a normal $p_i$-subgroup of $(T_i(P_1\cap
    K_i))/K_{i-1}$, we have
 that the elements of $T_i/K_{i-1}$ commute with the elements of
 $((P_1\cap K_i)K_{i-1})/K_{i-1}$. Since $h\notin K_{i-1}$, there exists $x\in X$
 such that $h(K_{i-1}(x))\neq K_{i-1}(x)$.

 Recall that $\mathcal{S}_{i-1}=\{ K_{i-1}(y)\mid y\in X\}$ is a system of imprimitivity of $\mathcal{G}(X,r)$ and
 $|\mathcal{S}_{i-1}|=p_i\cdots p_n$. Since $T_i/K_{i-1}$ is a normal non-trivial $p_i$-subgroup of $\mathcal{G}(X,r)$
 and $\mathcal{G}(X,r)$ acts transitively on $\mathcal{S}_{i-1}$, we have that the orbit
 $\mathcal{D}=\{tK_{i-1}(x)\mid t\in T_i\}$ of $K_{i-1}(x)$ by the action of $T_i$ has cardinal $p_i$, that is
 $|\mathcal{D}|=p_i$. Since $K_i=\{ g\in \mathcal{G}(X,r)\mid gT_i(y)=T_i(y)$ for all $y\in X\}$, $K_i$ acts on $\mathcal{D}$.
 Hence the map $\nu\colon K_i\longrightarrow \Sym_{\mathcal{D}}$, defined by
 $\nu(g)(tK_{i-1}(x))=gtK_{i-1}(x)$, is a well-defined homomorphism of groups.
 Then $\nu(K_i)$ is a solvable transitive group of degree $p_i$, thus $\nu(K_i)$ is primitive and by Remark \ref{Galois},
 $\nu(K_i)\cong \Z/(p_i)\rtimes \langle \xi\rangle$ for some $\xi\in\Aut(\Z/(p_i))$.
 Furthermore, $\nu(T_i)\cong \Z/(p_i)$ is self-centralizing
 in $\nu(K_i)$. Since $h(K_{i-1}(x))\neq K_{i-1}(x)$, we have that $\nu(h)\neq 1$
 and has order a power of $p_1$, because $h\in P_1$. Therefore, $h\in K_i$
 implies that $\nu(h)\in \nu(K_i)\setminus \nu(T_i)$.
 Since $hK_{i-1}tK_{i-1}=tK_{i-1}hK_{i-1}$ for all
 $t\in T_i$ and $K_{i-1}\subseteq \ker(\nu)$, we have that
 $\nu(h)\nu(t)=\nu(t)\nu(h)$ for all $t\in T_i$,
 a contradiction because $\nu(T_i)$ is self-centralizing in $\nu(K_i)$.
Therefore, the result follows.
    \end{proof}

\section{A construction} \label{construction}

Note that, by Corollary \ref{cmulti}, if $(X,r)$ is an
indecomposable multipermutation solution of the YBE of cardinality
$p_1\cdots p_n$, where $p_1,\dots ,p_n$ are $n$ distinct prime
numbers, then $(X,r)$ has multipermutation level $\leq n$.
Moreover, since $\mathcal{G}(X,r)$ is a solvable transitive
subgroup of $\Sym_X$, by Theorem \ref{squarefreedegree},
$p_1\cdots p_n$ is a divisor of $|\mathcal{G}(X,r)|$. We conclude
with an example showing that indecomposable multipermutation
solutions of the YBE of cardinality $p_1\cdots p_n$ and
multipermutation level $n$ indeed exist.

\begin{example} {\rm
    Let $p_1,...,p_{n}$ be different primes.  For all $i=1,2,\dots ,n$, we shall construct
    inductively an indecomposable solution $(X_i,r_i)$ of the YBE of cardinality
    $p_1p_2\cdots p_{i}$, we shall prove that $\Ret(X_{j},r_j)=(X_{j-1},r_{j-1})$
    and $\mathcal{G}(X_j)\cong \Z/(p_j)^{|X_j|}\rtimes_{\alpha} \mathcal{G}(X_{j-1},r_{j-1})$
    for all $2\leq j\leq n$, as left braces, where
    $$\alpha\colon \mathcal{G}(X_{j-1},r_{j-1})\longrightarrow \Aut(\Z/(p_j)^{|X_j|})$$
    is defined by
    $$\alpha(g)((a_x)_{x\in X_{j-1}})=(a_{g^{-1}(x)})_{x\in X_{j-1}}$$
    for $g\in \mathcal{G}(X_{j-1},r_{j-1})$. Note that $\Z/(p_j)^{|X_j|}\rtimes_{\alpha} \mathcal{G}(X_{j-1},r_{j-1})=\Z/(p_j)\wr\mathcal{G}(X_{j-1},r_{j-1})$
    is the wreath product of the trivial brace $\Z/(p_j)$ by the left brace $\mathcal{G}(X_{j-1},r_{j-1})$ (see \cite[Section 3]{CedoSurvey}).
    In particular, $(X_n,r_n)$ will be an indecomposable, multipermutation solution of the YBE of cardinality $p_1\cdots p_n$ and multipermutation level $n$. Furthermore
    $$\mathcal{G}(X_n,r_n)\cong \Z/(p_n)\wr (\Z/(p_{n-1})\wr(\dots (\Z/(p_2)\wr\Z/(p_1))\dots))$$
    as left braces  and
    $|\mathcal{G}(X_n,r_n)|=p_1p_2^{|X_1|}\cdots p_n^{|X_{n-1}|}=p_1p_2^{p_1}\cdots p_n^{p_1\cdots p_{n-1}}$.

    Let $|X_1|=p_1$ and let
    $(X_1,r_1)$ be an indecomposable solution of the YBE. From Theorem~2.13
    in \cite{ESS} we know that
    $G_1 = \mathcal G (X_1,r_1) = \mathbb{Z}/(p_1)$.

    Suppose that $1<j\leq n$ and that we have constructed indecomposable solutions
    $(X_{i},r_{i})$ of the YBE of cardinality $p_1\cdots p_{i}$, for all $1\leq i<j$, such that if $i>1$, then $\Ret(X_{i},r_{i})=(X_{i-1},r_{i-1})$   and $\mathcal{G}(X_{i})\cong \Z/(p_{i})^{|X_{i-1}|}\rtimes_{\alpha} \mathcal{G}(X_{i-1},r_{i-1})$, as left braces, where
    $$\alpha\colon \mathcal{G}(X_{i-1},r_{i-1})\longrightarrow \Aut(\Z/(p_{i})^{|X_{i-1}|})$$
    is defined  by
    $$\alpha(g)((a_x)_{x\in X_{i-1}})=(a_{g^{-1}(x)})_{x\in X_{i-1}},$$
    for all $g\in \mathcal{G}(X_{i-1},r_{i-1})$ and $a_x\in \Z/(p_i)$.
    Let $X_{j}=\Z/(p_j)\times X_{j-1}$. We define $r_j\colon X_j\times X_j\longrightarrow X_j\times X_j$ by
    $r_j((a,x),(b,y))=(\sigma_{(a,x)}(b,y),\sigma^{-1}_{\sigma_{(a,x)}(b,y)}(a,x))$ for all $(a,x),(b,y)\in X_j$, where
    $\sigma_{(a,x)}(b,y)=(b+\delta_{x,\sigma_x(y)},\sigma_x(y))$
    and the permutations $\sigma_x$ correspond to the solution $(X_{j-1},r_{j-1})$.
    Then $\sigma_{(a,x)}^{-1} ((b,y)) = (b-\delta_{x,y},\sigma_{x}^{-1}(y))$, so that for $(c,z)\in X_{j}$
    we get
    \begin{align*}
        \sigma_{(a,x)} \sigma_{\sigma_{(a,x)}^{-1} (b,y)} (c,z) & =
        \sigma_{(a,x)} \sigma_{(b-\delta_{x,y},\sigma_{x}^{-1} (y))}
        (c,z)\\
        & =  \sigma_{(a,x)}(c+\delta_{\sigma_{x}^{-1}(y),\sigma_{\sigma_{x}^{-1} (y)} (z)}, \sigma_{\sigma_{x}^{-1} (y)} (z))\\
        & = ( c+\delta_{\sigma_{x}^{-1}(y),\sigma_{\sigma_{x}^{-1} (y)} (z)}+\delta_{x,\sigma_x \sigma_{\sigma_{x}^{-1} (y)} (z)},\sigma_x \sigma_{\sigma_{x}^{-1} (y)} (z)) \\
        & = ( c+\delta_{y,\sigma_x\sigma_{\sigma_{x}^{-1} (y)} (z)}+\delta_{x,\sigma_x \sigma_{\sigma_{x}^{-1} (y)} (z)},\sigma_x \sigma_{\sigma_{x}^{-1} (y)} (z)).
    \end{align*}
    Since $\sigma_x \sigma_{\sigma_{x}^{-1} (y)}=\sigma_y \sigma_{\sigma_{y}^{-1}
        (x)}$ (because $(X_{j-1},r_{j-1})$ is a solution of the
    YBE by the inductive hypothesis), by a symmetric calculation this is equal to
    $$\sigma_{(b,y)} \sigma_{\sigma_{(b,y)}^{-1} (a,x)} (c,z).$$
    Hence, $(X_{j}, r_{j})$ is a solution of the YBE. Note that by the definition of the $\sigma_{(a,x)}$ we have that $\sigma^{-1}_{(a,x)}=\sigma^{-1}_{(b,y)}$ if and only if $\sigma^{-1}_x=\sigma^{-1}_y$ and $\delta_{x,z}=\delta_{y,z}$ for all $z\in X_{j-1}$. Hence $\sigma^{-1}_{(a,x)}=\sigma^{-1}_{(b,y)}$ if and only if $x=y$. Let $f\colon\Ret(X_j,r_j)\longrightarrow (X_{j-1},r_{j-1})$ be the map defined by $f(\overline{(a,x)})=x$ for all $(a,x)\in X_j$. We have seen that $f$ is bijective. Note that
    \begin{align*}
        f(\sigma_{\overline{(a,x)}}(\overline{(b,y)})) & =f(\overline{\sigma_{(a,x)}(b,y)})\\
       & =f(\overline{(b+\delta_{(x,\sigma_x(y))},\sigma_x(y))})\\
        &=\sigma_x(y)=\sigma_{f(\overline{(a,x)})}(f(\overline{(b,y)}))
    \end{align*}
for all $(a,x),(b,y)\in X_j$. Hence $f$
     is an isomorphism of solutions.

     Let $G_j=\Z/(p_{j})^{|X_{j-1}|}\rtimes_{\alpha} \mathcal{G}(X_{j-1},r_{j-1})$, where
     $$\alpha\colon \mathcal{G}(X_{j-1},r_{j-1})\longrightarrow \Aut(\Z/(p_{j})^{|X_{j-1}|})$$
     is defined  by
     $$\alpha(g)((a_x)_{x\in X_{j-1}})=(a_{g^{-1}(x)})_{x\in X_{j-1}},$$
     for all $g\in \mathcal{G}(X_{j-1},r_{j-1})$ and $a_x\in \Z/(p_j)$.
     Let $\phi\colon G_j\longrightarrow \Sym_{X_j}$ be the map defined by

        \begin{equation*}
     \phi((a_x)_{x\in X_{j-1}}, g)(b,y)=(b+a_{g(y)},g(y))
     \end{equation*}
     for all $((a_x)_{x\in X_{j-1}}, g)\in G_j$ and all $(b,y)\in X_j$. Note that
     \begin{align*}
        \phi(((a_x)_{x\in X_{j-1}}, g)&((b_x)_{x\in X_{j-1}}, h))(c,z)\\
        &=\phi(((a_x)_{x\in X_{j-1}}+(b_{g^{-1}(x)})_{x\in X_{j-1}}, gh))(c,z) \nonumber\\
        &=\phi(((a_x+b_{g^{-1}(x)})_{x\in X_{j-1}}, gh))(c,z)  \nonumber \\
       & =(c+a_{gh(z)}+b_{g^{-1}(gh(z))},gh(z))=(c+a_{gh(z)}+b_{h(z)},gh(z))  \nonumber \\
        &=\phi((a_x)_{x\in X_{j-1}}, g)(c+b_{h(z)},h(z))  \nonumber \\
       & =\phi((a_x)_{x\in X_{j-1}}, g)\phi((b_x)_{x\in
       X_{j-1}},h)(c,z).  \nonumber
     \end{align*}
     Hence $\phi$ is a homomorphism of groups. Let $((a_x)_{x\in X_{j-1}}, g)\in \ker(\phi)$. Then
     $$((a_x)_{x\in X_{j-1}}, g)(b,y)=(b+a_{g(y)},g(y))=(b,y)$$
     for all $(b,y)\in X_j$. Thus $g=\id$ and $a_x=0$ for all $x\in X_{j-1}$. It follows that $\phi$ is injective. Clearly $\phi(G_j)$ acts trasitively on $X_j$. Note that
     $$G_j=\langle ((\delta_{x,y})_{x\in X_{j-1}},\sigma_y)\mid y\in X_{j-1}\rangle$$
     and
     \begin{equation} \label{phi}
     \phi((\delta_{x,y})_{x\in X_{j-1}},\sigma_y)(c,z)=(c+\delta_{\sigma_y(z),y},\sigma_y(z))=\sigma_{(b,y)}(c,z)
     \end{equation}
     for all $b\in \Z/(p_j)$, $y\in X_{j-1}$ and $(c,z)\in X_j$. Hence $\phi(G_j)=\mathcal{G}(X_j,r_j)$. Since $\phi(G_j)$ acts transitively on $X_j$, we have that $(X_j,r_j)$ is indecomposable. Note that
     \begin{align} \label{lam}
     \lambda_{((\delta_{x,y})_{x\in X_{j-1}},\sigma_y)}&((\delta_{x,z})_{x\in X_{j-1}},\sigma_z)\\
        &=-((\delta_{x,y})_{x\in X_{j-1}},\sigma_y)+((\delta_{x,y})_{x\in X_{j-1}},\sigma_y)((\delta_{x,z})_{x\in X_{j-1}},\sigma_z)  \nonumber \\
        &=-((\delta_{x,y})_{x\in X_{j-1}},\sigma_y)+((\delta_{x,y})_{x\in X_{j-1}}+(\delta_{\sigma^{-1}_y(x),z})_{x\in X_{j-1}},\sigma_y\sigma_z) \nonumber \\
        &=((\delta_{\sigma^{-1}_y(x),z})_{x\in X_{j-1}},- \sigma_y+\sigma_y\sigma_z) \nonumber  \\
       & =((\delta_{x,\sigma_y(z)})_{x\in X_{j-1}},\lambda_{\sigma_y}(\sigma_z))\nonumber  \\
        &=((\delta_{x,\sigma_y(z)})_{x\in
        X_{j-1}},\sigma_{\sigma_y(z)}), \nonumber
        \end{align}
     the last equality by (\ref{Lemma 2.1}).  Hence, by the definition of $\sigma_{(a,y)}$ and
     applying (\ref{phi}) and (\ref{lam}),  we get
     \begin{align*}\phi(\lambda_{((\delta_{x,y})_{x\in X_{j-1}},\sigma_y)}((\delta_{x,z})_{x\in X_{j-1}},\sigma_z))&=
     \sigma_{(b,\sigma_y(z))}=\sigma_{\sigma_{(a,y)}(b-\delta_{y,\sigma_y(z)},z)}\\
        &=\lambda_{\sigma_{(a,y)}}(\sigma_{(b-\delta_{y,\sigma_y(z)},z)})\\
        &=\lambda_{\phi((\delta_{x,y})_{x\in X_{j-1}},\sigma_y)}(\phi((\delta_{x,z})_{x\in X_{j-1}},\sigma_z)),
        \end{align*}
        the third equality by (\ref{Lemma 2.1}),
     and thus $G_j\cong \mathcal{G}(X_j,r_j)$ as left braces. Therefore the claim follows by induction.
    }
    \end{example}

 We conclude with the following comments. Let $p_1,p_2$ be
prime numbers such that $p_2\equiv 1\mod p_1$. Let $g\in
\Z/(p_2)\setminus\{ 0\}$ be an element of multiplicative order
$p_1$. Let $G=\Z/(p_2)\rtimes_{\alpha}\Z/(p_1)$ be the semidirect
product of the trivial braces $\Z/(p_2)$ and $\Z/(p_1)$, where
$\alpha\colon \Z/(p_1)\longrightarrow \Aut(\Z/(p_2))$ is defined
by $\alpha(a)(b)=g^{a}b$, for all $a\in\Z/(p_1)$ and
$b\in\Z/(p_2)$. Then, by \cite[Theorem 3.3]{Ramirez}, there exists
an indecomposable multipermutation solution $(X,r)$ of the YBE of
cardinality $p_1p_2$, multipermutation level $2$ and such that
$\mathcal{G}(X,r)\cong G$ as left braces.

However, this cannot be generalized for more than two primes, that
is, if $p_1,p_2,p_3$ are three distinct prime numbers and $(X,r)$
is an indecomposable multipermutation solution of the YBE of
cardinality $p_1p_2p_3$ and multipermutation level $3$, then
$|\mathcal{G}(X,r)|>p_1p_2p_3$. Indeed, let $P_i$ be the Sylow
$p_i$-subgroup of the additive group of the left brace
$\mathcal{G}(X,r)$. Suppose that $|P_i|=p_i$ for all $i$. By
Theorem \ref{multi}, we may assume that $P_3$ and $P_3P_2$ are
ideals of the left brace $\mathcal{G}(X,r)$. Then
$\mathcal{G}(X,r)=P_3(P_2P_1)$ is an inner semidirect product of
the ideal $P_3$ and the left ideal $P_2P_1$. Hence,
(\ref{semidirect1}) implies  that $\lambda_{c}(ab)=ab$, for all
$a\in P_1$, $b\in P_2$ and $c\in P_3$. Since $P_3$ is a trivial
brace, we have that $P_3\subseteq \soc(\mathcal{G}(X,r))$. Since
$P_3P_2$ is an ideal of $\mathcal{G}(X,r)$ and
$P_2P_1\cong\mathcal{G}(X,r)/P_3$, we have that $P_2$ is an ideal
of $P_2P_1$. Thus $aba^{-1}b^{-1}\in P_2$, for all $a\in P_1$ and
$b\in P_2$. Since $|P_3|=p_3$, $\Aut(P_3)$ is  abelian. Hence, as
$P_3$ is a left ideal of $\mathcal{G}(X,r)$, we have that
    \begin{equation}\label{com1}
        \lambda_{aba^{-1}b^{-1}}(c)=\lambda_a\lambda_b\lambda^{-1}_a\lambda^{-1}_b(c)=c
    \end{equation}
for all $a\in P_1$, $b\in P_2$ and $c\in P_3$. Since $P_2$ and
$P_3$ are trivial braces and left ideals of $\mathcal{G}(X,r)$ we
have
\begin{equation}\label{com2}
    \lambda_{aba^{-1}b^{-1}}(a')=\lambda_a\lambda_b\lambda^{-1}_a(\lambda^{-1}_b(a'))=
    \lambda_a (\lambda_b\lambda^{-1}_b(a'))= \lambda_b\lambda^{-1}_b(a') = a'
\end{equation}
and similarly
\begin{equation}\label{com3}
    \lambda_{aba^{-1}b^{-1}}(b')=\lambda_a\lambda_b\lambda^{-1}_a\lambda^{-1}_b(b')=
    \lambda_a\lambda_b(\lambda^{-1}_a(b)) =
    \lambda_a\lambda^{-1}_a(b')=b'
\end{equation}
for all $a,a'\in P_1$ and $b,b'\in P_2$. By (\ref{com1}),
(\ref{com2}) and (\ref{com3}), we get that
$aba^{-1}b^{-1}\in\soc(\mathcal{G}(X,r))\cap P_2$ for all $a\in
P_1$ and $b\in P_2$. Thus  either $[P_1,P_2]$ is the trivial
group, so $P_1P_2$ is abelian, or $\soc(\mathcal{G}(X,r))$
contains an element of order $p_2$, so it must contain $P_2$.
Since the map $f\colon \Ret(X,r)\longrightarrow \mathcal{G}(X,r)$
defined by $f(\overline{x})=\sigma_x$, for all $x\in X$, is an
injective homomorphism of solutions from $\Ret(X,r)$ into the
solution associated to the left brace $\mathcal{G}(X,r)$ and
$\mathcal{G}(X,r)=\langle \sigma_x\mid x\in X\rangle$, there is an
injective homomorphism of solutions $\overline{f}\colon
\Ret^2(X,r)\longrightarrow
\mathcal{G}(X,r)/\soc(\mathcal{G}(X,r))$ from $\Ret^2(X,r)$ into
the solution associated to the left brace
$\mathcal{G}(X,r)/\soc(\mathcal{G}(X,r))$.  Note that if $P_2P_1$
is abelian, then $P_2P_1$ is an inner direct product of its ideals
$P_1$ and $P_2$. Since $P_1$ and $P_2$ are trivial braces,
(\ref{semidirect1}) implies that $P_2P_1$ is a trivial brace.
Hence, as $P_3\subseteq \soc(\mathcal{G}(X,r))$, we get that
$\mathcal{G}(X,r)/\soc(\mathcal{G}(X,r))$ is a trivial brace, in
this case. If $P_3P_2\subseteq \soc(\mathcal{G}(X,r))$, then, as
$P_1$ is a trivial brace, we also have that
$\mathcal{G}(X,r)/\soc(\mathcal{G}(X,r))$ is a trivial brace.
Hence, in both cases $\Ret^2(X,r)$ is a trivial solution. Since
the natural map $(X,r)\longrightarrow \Ret^2(X,r)$ is an
epimorphism of solutions, by Lemma \ref{CCP}, $\Ret^2(X,r)$ is
indecomposable and thus it has cardinality $1$, a contradiction
because $(X,r)$ has multipermutation level $3$. Therefore,
$|\mathcal{G}(X,r)|>p_1p_2p_3$.

Now it is easy to see that for every integer $n>2$ and $n$
distinct  prime numbers $p_1,\dots p_n$, if $(X,r)$ is an
indecomposable multipermutation solution of the YBE of cardinality
$p_1\cdots p_n$ and multipermutation level $n$, then
$|\mathcal{G}(X,r)|>p_1\cdots p_n$.

\vspace{30pt}
 \noindent \begin{tabular}{llllllll}
  F. Ced\'o && J. Okni\'{n}ski \\
 Departament de Matem\`atiques &&  Institute of
Mathematics \\
 Universitat Aut\`onoma de Barcelona &&   University of Warsaw\\
08193 Bellaterra (Barcelona), Spain    &&  Banacha 2, 02-097 Warsaw, Poland \\
 cedo@mat.uab.cat && okninski@mimuw.edu.pl\\

\end{tabular}


\begin{thebibliography}{99}
\itemsep=-2pt
\bibitem{BCJ}  D. Bachiller, F. Ced\'o and E. Jespers, Solutions of
the Yang--Baxter equation associated with a left brace,  J.
Algebra 463 (2016), 80--102.
\bibitem{BCV} D. Bachiller, F. Ced\'o and L. Vendramin,  A characterization
of finite multipermutation solutions of the Yang--Baxter equation,
Publ. Mat. 62 (2018), 641--649.
\bibitem{Baxter} R. J. Baxter, Partition function of the eight-vertex lattice model,
 Annals of Physics, 70 (1972), 193--228.
\bibitem{Mora-Sastriq}    S. Camp-Mora and R. Sastriques, A
criterion for decomposabilty in QYBE, International Mathematics
Research Notices, rnab357, https://doi.org/10.1093/imrn/rnab357.
\bibitem{CCP} M. Castelli, F. Catino and G. Pinto, Indecomposable involutive set-theoretic
solutions of the Yang--Baxter equation, J. Pure and Appl. Algebra,
223 (2019), 4477--4493.
\bibitem{CCP2020} M. Castelli, F. Catino and G. Pinto, Indecomposable involutive
set-theoretic solutions of the Yang--Baxter equation and
orthogonal dynamical extensions of cycle sets, arXiv:2011.10083.
\bibitem{CPR} M. Castelli, G. Pinto and W. Rump,  On the indecomposable involutive set-theoretic
solutions of the Yang--Baxter equation of prime-power size, Comm.
Algebra 48 (2020), 1941--1955.
\bibitem{CedoSurvey} F. Ced\'o,  Left braces: solutions of the Yang--Baxter equation,
Adv. Group Theory Appl. 5 (2018), 33--90.
\bibitem{CJKVAV} F. Ced\'o, E. Jespers, {\L}. Kubat, A. Van Antwerpen and C. Verwimp,
On various types of nilpotency of the structure monoid and group
of a set-theoretic solution of the Yang--Baxter equation, J. Pure
Appl. Algebra 227 (2023), no. 2, Paper No. 107194, 38 pp.
\bibitem{CJOComm} F. Ced\'o, E. Jespers and J. Okni\'nski, Braces and the Yang--Baxter
equation, Commun. Math. Phys. 327 (2014), 101--116.
\bibitem{CJOprimit} F. Ced\'o, E. Jespers
and J. Okni\'{n}ski, Primitive set-theoretic solutions of the
Yang--Baxter equation, Comm. Contemp. Math. Vol. 24, No. 9 (2022)
2150105, 10 pp.
\bibitem{CO21} F. Ced\'o and J. Okni\'{n}ski, Constructing finite simple solutions of
the Yang--Baxter equation, Adv. Math. 391 (2021), 107968, 39 pp.
\bibitem{CO22}  F. Ced\'o and J. Okni\'{n}ski, New simple solutions of the Yang--Baxter
equation and solutions associated to simple left braces, J.
Algebra 600 (2022), 125--151.
\bibitem{drinfeld} V. G. Drinfeld, On some unsolved problems in quantum
group theory. Quantum Groups, Lecture Notes Math. 1510,
Springer-Verlag, Berlin, 1992, 1--8.
\bibitem{ESS} P. Etingof, T. Schedler and A.
Soloviev, Set-theoretical solutions to the quantum Yang--Baxter
equation, Duke Math. J. 100 (1999), 169--209.
\bibitem{GIC}   T. Gateva-Ivanova and P. Cameron, Multipermutation solutions
of the Yang--Baxter equation, Commun. Math. Phys. 309 (2012),
583--621.
\bibitem{GIVdB} T. Gateva-Ivanova and M. Van den Bergh,
Semigroups of $I$-type, J. Algebra 206 (1998), 97--112.
\bibitem{H} B. Huppert, Endliche Gruppen I, Spriger-Verlag, Berlin
1967.
\bibitem{Jedl-Pilit-Zam} P. Jedlicka, A. Pilitowska and  A.
Zamojska-Dzienio, Indecomposable involutive solutions of the
Yang--Baxter equation of multipermutation level 2 with abelian
permutation group, Forum Math. 33 (2021), 1083--1096.
\bibitem{Jedl-Pilit-Zam2021} P. Jedlicka, A. Pilitowska and  A.
Zamojska-Dzienio, Cocyclic braces and indecomposable cocyclic
solutions of the Yang--Baxter equation,  Proc. Amer. Math.
Soc. 150 (2022), 4223--4239.
\bibitem{K} C. Kassel, Quantum Groups, Springer-Verlag, 1995.
\bibitem{lebed_and_co} V. Lebed, S. Ramirez and L.
Vendramin, Involutive Yang--Baxter: cabling, decomposability,
Dehornoy class, arXiv: 2209.02041.
\bibitem{Ramirez} S. Ram\'{i}rez, Indecomposable solutions of the Yang--Baxter
equation with permutation group of sizes $pq$ and $p^2q$,
arXiv: 2208.06741.
\bibitem{ramirez-vendramin}  S. Ram\'{i}rez and L. Vendramin, Decomposition
Theorems for Involutive Solutions to the Yang--Baxter Equation,
International Mathematics Research Notices, Vol. 2022, Issue 22,
Nov. 2022, pages 18078--18091.
\bibitem{Rump1} W. Rump, A decomposition theorem for square-free
unitary solutions of the quantum Yang--Baxter equation, Adv. Math.
193 (2005), 40--55.
\bibitem{R07} W. Rump,
Braces, radical rings, and the quantum Yang--Baxter equation, J.
Algebra 307 (2007), 153--170.
\bibitem{RumpSurvey} W. Rump,  The brace of a classical group, Note Mat. 34 (2014),
no. 1, 115--144.
\bibitem{rump2020} W. Rump, Classification of indecomposable involutive set-theoretic
solutions to the Yang--Baxter equation, Forum Math. 32 (2020), no.
4, 891--903.
\bibitem{Rump2020}  W. Rump, One-generator braces and indecomposable set-theoretic
solutions to the Yang--Baxter equation, Proc. Edinb. Math. Soc.
(2) 63 (2020), no. 3, 676--696.
\bibitem{rump2022}  W. Rump, Classification of non-degenerate involutive
set-theoretic solutions to the Yang--Baxter equation with
multipermutation level two, Algebras and Representation Theory 25
(2022), 1293--1307.
\bibitem{Smok} A. Smoktunowicz, On Engel groups, nilpotent groups,
rings, braces and the Yang--Baxter equation, Trans. Amer. Math.
Soc. 370 (2018), 6535--6564.
\bibitem{Smokt}  A. Smoktunowicz, A
note on set-theoretic solutions of the Yang--Baxter equation, J.
Algebra 500 (2018), 3--18.
\bibitem{SmokSmok} Agata Smoktunowicz and Alicja Smoktunowicz, Set-theoretic solutions
of the Yang--Baxter equation and new classes of $R$-matrices,
Linear Algebra and its Applications 546 (2018), 86--114.
\bibitem{V} L. Vendramin, Extensions of set-theoretic solutions of the Yang--Baxter
equation and a conjecture of Gateva-Ivanova, J. Pure Appl. Algebra
220 (2016), 2064--2076.
\bibitem{Yang} C.-N. Yang, Some exact results for the many-body problem in one
dimension with repulsive delta-function interaction, Physical
Review Letters, 19 (1967), 1312.
\end{thebibliography}
\end{document}